\documentclass[letterpaper]{amsart}

\usepackage{hyperref}
\usepackage{mathrsfs}
\usepackage{amssymb}
\usepackage[curve]{xypic}

\title{Symmetric cubical sets}
\author{Samuel B. Isaacson}
\address{Department of Mathematics\\
The University of Western Ontario\\
London ON N6A 5B7 \\
Canada}
\email{sisaacso@uwo.ca}

\numberwithin{equation}{section}

\theoremstyle{plain}

\newtheorem{theorem}{Theorem}[section]
\newtheorem{lemma}[theorem]{Lemma}
\newtheorem{proposition}[theorem]{Proposition}
\newtheorem{corollary}[theorem]{Corollary}

\theoremstyle{definition}
\newtheorem{definition}[theorem]{Definition}
\newtheorem{example}[theorem]{Example}

\theoremstyle{remark}
\newtheorem{remark}[theorem]{Remark}

\newcommand{\adjoint}{\dashv}
\newcommand{\Cat}{\mathbf}
\newcommand{\Op}{\mathrm{op}}
\newcommand{\Nerve}{\mathrm{N}}

\newcommand{\sqparen}[1]{[\![#1]\!]}

\newcommand{\sSet}{s\Cat{Set}}
\newcommand{\qSet}{q\Cat{Set}}
\newcommand{\qsSet}{q_\Sigma\Cat{Set}}

\DeclareMathOperator{\QQ}{\mathscr{Q}}
\DeclareMathOperator{\qsigma}{\mathscr{Q}_\Sigma}
\DeclareMathOperator{\id}{id}
\DeclareMathOperator{\Hom}{Hom}

\DeclareMathOperator{\Ob}{ob}
\DeclareMathOperator*{\colim}{colim}
\DeclareMathOperator{\sk}{sk}
\DeclareMathOperator{\ck}{ck}
\DeclareMathOperator{\Ar}{ar}
\DeclareMathOperator{\Ho}{Ho}
\DeclareMathOperator*{\hocolim}{hocolim}

\DeclareMathOperator{\Aut}{Aut}
\DeclareMathOperator{\Sing}{Sing}

\DeclareMathOperator{\Stab}{Stab}
\DeclareMathOperator{\Cell}{Cell}
\DeclareMathOperator{\inj}{Inj}
\DeclareMathOperator{\Cyl}{Cyl}
\DeclareMathOperator{\Ch}{Ch}
\DeclareMathOperator{\CS}{cs}
\DeclareMathOperator{\CQ}{cq}

\begin{document}

\nocite{isaacson,goerssjardine,homotopicalalgebra}

\begin{abstract}
  We introduce a new cubical model for homotopy types.  More
  precisely, we'll define a category $\qsigma$ with the following
  features: $\qsigma$ is a \textsc{prop} containing the classical box
  category as a subcategory; the category $\qsSet$ of presheaves of
  sets on $\qsigma$ models the homotopy category; and combinatorial
  symmetric monoidal model categories with cofibrant unit have
  homotopically well behaved $\qsSet$ enrichments.
\end{abstract}

\maketitle

\tableofcontents

\section{Introduction}
Classically we have two models for the homotopy category: one can
start with the category $\Cat{Top}$ of (compactly generated weak
Hausdorff) topological spaces, associate a \textsc{cw} approximation
$\gamma X\to X$ to each space $X$, and take $\Ho\Cat{Top}(X,Y)$ to be
the homotopy classes of maps between $\gamma X$ and $\gamma Y$.  The
Whitehead theorem implies that weak equivalences between \textsc{cw}
complexes are homotopy equivalences, so $\Ho\Cat{Top}$ is the
localization of $\Cat{Top}$ at the category of weak homotopy
equivalences.  Alternatively, one can use the category $\sSet$ of
simplicial sets, Kan approximations, and $\Delta[1]$-homotopy---with
the proviso that Kan approximations are ``on the right''---to
construct a homotopy category of simplicial sets.  The geometric
realization-singular set adjunction
\begin{equation*}
  \xymatrix{
    |{-}|:\sSet \ar@<0.5ex>[r] &
    \Cat{Top} : \Sing  \ar@<0.5ex>[l]}
\end{equation*}
is a Quillen equivalence: it descends to an equivalence of homotopy
categories (and preserves the homotopy types of mapping spaces).
Any homotopy-theoretic result true in $\Cat{Top}$ is thus true in
$\sSet$, and vice versa,
so one can view $\Cat{Top}$ and $\sSet$ as two presentations of the
same ($(\infty,1)$-) category, using whichever is more convenient for
the application at hand.

One advantage of simplicial sets is that the category $\sSet$ is a
presheaf topos, unlike $\Cat{Top}$ (the obvious disadvantage is that
almost no space comes ``in nature'' as a simplicial set, and many
geometric constructions rely on $\Cat{Top}$).  In fact, the category
$\Delta$ of finite nonempty totally ordered sets is not the only site
upon which we may model the homotopy category.  For example, the
cubical category $\QQ$---the category of posets $\{0 < 1\}^n$, $n\ge
0$ with maps those maps given by deleting coordinates or inserting
$0$s and $1$s---also models spaces via the associated category $\qSet$
of presheaves on $\QQ$.  This result, in the language of model
categories, is relatively recent.  Denis-Charles Cisinski and Georges
Maltsiniotis, building on conjectures of Grothendieck, have given a
unified perspective of categorical homotopy theory and presheaf models
for the homotopy category in
\cite{cisinskithesis,maltsiniotis,pursuingstacks} (see
\cite{jardinecubes} as well)---one side benefit is a straightforward
demonstration that $\qSet$ is a model for the homotopy category.

One advantage of the cubical category $\QQ$ is that the product of two
cubes is again a cube: in $\sSet$, the product of two representable
functors (i.e., two simplices) is not itself representable.  This
considerably simplifies the project of finding a spatial enrichment in
an arbitrary homotopical category: the ``$n$-cubes'' of a cubical
mapping space are simply the $n$-fold homotopies.  Of course, cubical
sets come with their own disadvantages: without adding extra
degeneracies, the analogous Dold-Kan correspondence fails
\cite{brownhigginsdoldkan}; and the convolution monoidal structure is
not symmetric.  In order to remedy these there is a menagerie of
cubical categories containing $\QQ$ as a subcategory
\cite{grandismauri}.  In this paper, we'll add one more category
$\qsigma$ to the zoo, with some useful features:
\begin{enumerate}
\item $\qsigma$ is symmetric monoidal---in fact, it is a \textsc{prop}
  in $\Cat{Set}$---and hence the category $\qsSet =
  \Cat{Set}^{\qsigma^\Op}$ is symmetric monoidal.
\item There are left Quillen equivalences $i_!:\qSet \to \qsSet$ and
  $|{-}|_\Sigma:\qsSet \to \Cat{Top}$.  These are strong monoidal and
  strong symmetric monoidal, respectively (Theorem \ref{theorem:quillenequivalences})
\item Any combinatorial symmetric monoidal model category with
  cofibrant unit may be enriched over $\qsSet$.  (Theorem \ref{theorem:final})
\end{enumerate}
In a future paper, we'll show how (3) may be leveraged to show that
combinatorial monoidal model categories may be realized as
localizations of categories of presheaves of spaces with a convolution
model structure, giving a special case of a theorem of Daniel Dugger's
in \cite{dugger1}.

The plan of this paper is as follows.  In part \ref{part:one}, we'll
describe the category $\qsigma$ as a \textsc{prop} and show how to
lift the model structure on $\qSet$ to $\qsSet$ via a left Kan
extension $i_!:\qSet \to \qsSet$.  Along the way, we'll describe how
to decompose presheaves in $\qsSet$ as colimits of their skeleta; this
requires some careful combinatorial work, but is straightforward.
This relies in large part on the methods of Cisinski
\cite{cisinskithesis}.  In part \ref{part:two}, we'll discuss $\qsSet$
enrichments and some miscellaneous results: we'll show that the
cubical mapping spaces associated with $\qsSet$ enrichments have the
correct homotopy type, and we'll use an argument of Schwede and
Shipley to show that every combinatorial symmetric monoidal model
category with cofibrant unit has a $\qsSet$ enrichment.

\subsection{Acknowledgments}
This paper is a condensation of the first two chapters of my PhD
dissertation \cite{isaacson}, written under the supervision of Michael
J.~Hopkins.  I'd like to thank him along with Reid Barton, Clark
Barwick, Samik Basu, Mark Behrens, Andrew Blumberg, Tony Elmendorf,
Phil Hirschhorn, Dan Kan, Jacob Lurie, Peter May, and Haynes Miller,
for countless hours of discussion and criticism.

\part{Modeling spaces}\label{part:one}

\section{Day convolution}
Before we introduce cubical sets, we briefly review some enriched
category theory and introduce some notation.  Suppose $\mathscr{V}$ is
a closed symmetric monoidal category with all small limits and
colimits and $\mathscr{I}$ a small $\mathscr{V}$-category
\cite{kelly}.  Write $[{-}]:\mathscr{I}\to \widehat{\mathscr{I}}$ for
the enriched Yoneda embedding, $\widehat{\mathscr{I}} =
\mathscr{V}^{\mathscr{I}^\Op}$.  Recall that $[{-}]$ displays
$\widehat{\mathscr{I}}$ as the ``free cocompletion'' of
$\mathscr{\mathscr{I}}$: the category of indexed colimit-preserving
functors out of $\widehat{\mathscr{I}}$ is equivalent to the category
of functors out of $\mathscr{\mathscr{I}}$ \cite{kelly,maclane}.  Now
suppose $(\mathscr{I},\otimes,e)$ is monoidal.  By the universal
property of the Yoneda embedding we just mentioned, there is a
monoidal structure on $\widehat{\mathscr{I}}$, unique up to unique
isomorphism, with the following properties:
\begin{enumerate}
\item $[i] \otimes [j] = [i\otimes j]$ for $i,j\in \mathscr{I}$.
\item ${-} \otimes {-}$ is cocontinuous (i.e.,preserves all indexed
  colimits) in each variable.
\end{enumerate}
The canonical presentation of a presheaf as a colimit of
representables gives the coend formula
\begin{equation}
  (X \otimes Y)_k \cong \int^{i,j\in \mathscr{C}} \mathscr{I}(k, i\otimes
  j) \otimes (X_i \otimes Y_j).
\end{equation}
The unit is the representable presheaf $[e]$.  For fixed
$X\in\widehat{\mathscr{I}}$, the functors $X\otimes{-}$ and
${-}\otimes X$ both have right adjoints.  If $\mathscr{I}$ is
symmetric, then the product on $\widehat{\mathscr{I}}$ is
closed symmetric monoidal.  The hom functor $[{-},{-}]$ is given by
the end
\begin{equation}
  [X,Y]_i = \int_{\ell\in I} \mathscr{V}(X_\ell, Y_{i\otimes \ell}).
\end{equation}
The product $\otimes$ on $\widehat{\mathscr{I}}$ is known as \emph{Day
  convolution;} it was introduced in Day's thesis
\cite{dayconvolution}.  Im and Kelly prove the following result in
\cite{imkelly}; it is an application of the Yoneda lemma.
\begin{proposition}\label{prop:imkelly}
  Suppose $(\mathscr{C},\otimes,e)$ is a monoidal
  $\mathscr{V}$-category with small indexed colimits so that
  ${-}\otimes{-}$ preserves indexed colimits in each variable.  Given
  a strong monoidal functor $F:\mathscr{I}\to\mathscr{C}$, the unique
  cocontinuous extension $\widehat F:\widehat{\mathscr{I}} \to
  \mathscr{C}$ is strong monoidal.  If $\mathscr{C}$ and $\mathscr{I}$
  are symmetric monoidal categories and $F$ is symmetric strong
  monoidal, then $\widehat{F}$ is symmetric strong monoidal as well.
\end{proposition}

\section{The symmetric cubical site}
Historically, there are several cubical categories, each generated by
a selection of face, degeneracy and possibly symmetry maps.  Grandis
and Mauri give a zoology of cubical sites in \cite{grandismauri}; our
site $\qsigma$, defined below, is a novel addition.  The cubical
category with the fewest maps has as objects the posets $\{0 < 1\}^n$,
$n\ge 0$; a map $\{0 < 1\}^n \to \{0 < 1\}^m$ may erase coordinates
(degeneracies) and insert $0$ or $1$, but may \emph{not} repeat
coordinates or change their order.  This is the classical ``box
category''; we denote it by $\QQ$ and write $\qSet$ for the associated
category of presheaves of sets on $\QQ$.  We write $\square^n$ for the
representable presheaf $\QQ({-},\sqparen{n})$.  The category $\QQ$ has
a monoidal structure given by concatenation.  Viewed as a \textsc{pro}
\cite{boardmanvogt}, its algebras in a monoidal category
$(\mathscr{C},\otimes,e)$ are diagrams
\begin{equation}\label{eq:interval}
  \vcenter{\xymatrix@C=4pc{
      e \amalg e \ar[r] \ar@/_/[rr]_{\id\amalg\id} & I \ar[r] & e.
    }}
\end{equation}
This may be found in \cite{cisinskithesis}.  We'll call diagrams of
the shape \eqref{eq:interval} \emph{intervals.}

Brown and Higgins introduced in \cite{brownhiggins1,brownhiggins2} a
cubical site with an extra degeneracy called a ``connection''; the
connection maps are generated by the logical conjunction
\begin{equation}\label{eq:conjunction}
  {-} \wedge {-} : \{0 < 1\}^2 \to \{0 < 1\}
\end{equation}
with $x\wedge y = 1$ if and only if $x = y = 1$.  Since the term
``connection'' is widely established in differential geometry, we'll
call these maps \emph{conjunction} maps instead.  Imposing the
structure of a conjunction map on an interval motivates the following
definition:
\begin{definition}
  Suppose $(\mathscr{C},\otimes,e)$ is a monoidal category.  A
  \emph{cubical monoid} in $\mathscr{C}$ is a diagram
  \begin{equation}
    \vcenter{\xymatrix@C=4pc{
      e \amalg e \ar[r]^{d_0\amalg d_1} \ar@/_/[rr]_{\id\amalg\id} & I \ar[r]^s & e.
    }}
  \end{equation}
  together with a map $\mu:X\otimes X \to X$ so that
  \begin{enumerate}
  \item The map $\mu$ makes $X$ an associative monoid with unit $d_1$.
  \item The map $s$ is a monoid map.
  \item The map $d_0:e\to X$ is absorbing, i.e., the diagram
    \begin{equation}
      \vcenter{\xymatrix@C=4pc{
          X\otimes e \ar[d]_{s\otimes\id_e} \ar[r]^{\id_X \otimes d_0} &
          X\otimes X \ar[d]^\mu & e\otimes X \ar[l]_{d_0\otimes\id_X}
          \ar[d]^{\id_e \otimes s} \\
          e \ar[r]_{d_0} & X & e\ar[l]^{d_0}
        }}
    \end{equation}
    commutes.
  \end{enumerate}
  We'll sometimes abuse notation and simply say that $I$ is a cubical
  monoid.  A map of cubical monoids $I\to J$ is a map in $\mathscr{C}$
  commuting with all the structure data.  We write
  $\Cat{qMon}(\mathscr{C})$ for the category of cubical monoids in
  $\mathscr{C}$.
\end{definition}
We have the following examples:
\begin{example}
  Suppose $(\mathscr{C},\otimes,e)$ is a monoidal category.
  \begin{enumerate}
  \item The unit $e$ is a cubical monoid with $d_0 = d_1 = s = \id_e$
    and $\mu:e\otimes e\to e$ given by the coherence isomorphisms of
    $\mathscr{C}$.  This is the terminal cubical monoid in
    $\mathscr{C}$.
  \item The coproduct $e\amalg e$ is a cubical monoid with $d_0$ and
    $d_1$ given by the inclusion of each summand.  The multiplication
    $\mu$ and degeneracy $s$ are forced.  This is the initial cubical
    monoid in $\mathscr{C}$.
  \item The $1$-simplex $\Delta[1]\in \sSet$ is a cubical monoid
    via the conjunction map \eqref{eq:conjunction}.
  \item If $F:\mathscr{C} \to \mathscr{D}$ is lax symmetric monoidal
    and $I$ is a cubical monoid, then $FI$ is a cubical monoid; so,
    for example, the normalized chains on $\Delta[1]$ are a cubical
    monoid in chain complexes.
  \end{enumerate}
\end{example}

There is an alternative description of cubical monoids pointed out to
the author by Reid Barton.  Note that the category $([1],\wedge, 1)$
(here $[1] = \{0 < 1\}$) has the structure of a monoidal category.
Suppose $(\mathscr{C},\otimes,e)$ is a monoidal category in which
$\mathscr{C}$ has all small colimits and ${-}\otimes{-}$ preserves
colimits in each variable.  We may then equip the category
$\mathscr{C}^{[1]}$ of arrows in $\mathscr{C}$ with the Day
convolution model structure.  If $f:A\to B$ and $g:X\to Y$ are arrows
in $\mathscr{C}$, their product $f\odot g$ is the usual
pushout-product
\begin{equation}
  f\odot g:A \otimes Y \amalg_{A\otimes X} B\otimes X \to B\otimes Y.
\end{equation}
The unit is the unique map $\emptyset\to e$.  Note that $\emptyset\to
e$ and $\id_e:e\to e$ are both monoids in $\mathscr{C}^{[1]}$.
\begin{proposition}\label{prop:reidbarton}
  The category $\Cat{qMon}(\mathscr{C})$ of cubical monoids in
  $\mathscr{C}$ is equivalent to the category of monoids of the form
  $d_0:e\to X$ intervening in a diagram
  \begin{equation}
    \vcenter{\xymatrix{
        \emptyset \ar[d] \ar[r] &
        e \ar[r]^{\id_e} \ar[d]^{d_0} & e \ar[d]^{\id_e} \\
        e \ar[r]_{d_1} & X \ar[r]_s & e}}
  \end{equation}
   of monoids in $\mathscr{C}^{[1]}$.
\end{proposition}
Note that the condition that $d_0:e\to X$ be absorbing, in the
language of the product $\odot$, becomes the commutativity of the diagram
\begin{equation*}
  \vcenter{\xymatrix@C=4pc{
      X\otimes e \amalg_{e\otimes e} e\otimes X \ar[r]^{d_0\odot d_0} \ar[d] &
      X\otimes X \ar[d] \\
      e \ar[r]_{d_0} & X.}}
\end{equation*}
It is forced by requiring $d_0:e\to X$ to be a monoid in
$\mathscr{C}^{[1]}$.

\begin{definition}
  The category $\qsigma$ is the \textsc{prop} whose category of
  algebras in a symmetric monoidal category $(\mathscr{C},\otimes,e)$
  is the category $\Cat{qMon}(\mathscr{C})$.
\end{definition}
Since each cubical monoid yields an interval by forgetting structure,
we have a strict monoidal functor $i:\QQ \to \qsigma$.  This
definition of $\qsigma$ is fairly opaque; below, we will give fairly
explicit description of its maps.  Note that cubical monoids are
not abelian.  The \textsc{pro} whose algebras are cubical
monoids in an arbitrary monoidal category is straightforward to
describe: it is the cubical site obtained from $\QQ$ by adjoining
conjunction maps and the appropriate relations (see
\cite{grandismauri} or below).  The construction of $\qsigma$ is
analogous to the symmetrization of a non-$\Sigma$ operad
\cite{maygeometry}.  Note however that $\qsigma$ is not freely
generated by an operad, since it includes a $1$--$0$ operation
corresponding to the degeneracy $s:X\to e$.  This makes symmetrization
more complicated: as we'll see below, permutations can be moved past
the map $s$, but not past connections.

\subsection{The category $\qsigma$}
\begin{definition}
  Suppose $S$ is a set of symbols not containing $0$ or $1$.  A
  \emph{formal cubical product} on $S$ is either
  \begin{enumerate}
  \item an ordered conjunction of elements of $S$, none occurring more
    than once (i.e., a list of symbols in $S$ separated by $\wedge$); or
  \item the numeral $0$ or $1$.
  \end{enumerate}
  A \emph{formal cubical $(m,n)$-product} is an $n$-tuple of formal
  cubical products on $\{x_1,\dotsc,x_m\}$ so that no symbol $x_i$
  occurs more than once in its concatenation.  Write
  $\qsigma(\sqparen{m},\sqparen{n})$ for the set of all formal
  cubical $(m,n)$ products.  By convention,
  $\qsigma(\sqparen{m},\sqparen{0})$ is a single point.
\end{definition}

For example, the following are formal cubical $(3,2)$-products:
\begin{align*}
  & (x_1,0) && (1,x_3 \wedge x_2) && (1,1) && (x_1 \wedge x_3, x_2).
\end{align*}
However, $(x_1,x_1)$ is not a formal cubical product as the symbol
$x_1$ occurs more than once.

\begin{definition}
  The \emph{identity formal $(n,n)$-product} is the $n$-tuple
  $(x_1,\dotsc,x_n)$.  Suppose $X$ and $Y$ are formal cubical
  $(\ell,m)$- and $(m,n)$-products, respectively.  The
  \emph{composition} $Y\circ X$ is defined as follows:
  \begin{enumerate}
  \item Replace any occurrence of the symbol $x_i$ in $Y$ with the
    $i$th entry of $X$.
  \item Delete each occurrence of the symbol $1$ in each conjunction 
    of length at least two.
  \item Replace each conjunction containing $0$ with the numeral $0$.
  \end{enumerate}
\end{definition}
For example, we have the following compositions:
\begin{align*}
  (x_3,x_1\wedge x_2) \circ (0,x_1,x_5) &= (x_5,0) &
  (x_2\wedge x_1)\circ (x_1\wedge x_2,x_3) &= (x_3\wedge x_1\wedge
  x_2) \\
  (x_1\wedge x_2) \circ (1,1) &= (1) &
  (0,x_1\wedge x_4) \circ (x_{10},0,0,1,x_3) &= (0,x_{10}).
\end{align*}
This makes $\qsigma$ a category with objects $\sqparen{n}$, $n\ge 0$
and maps $\sqparen{m}\to\sqparen{n}$ formal cubical $(m,n)$-products.
We call $\qsigma$ the \emph{extended cubical category} and presheaves
on $\qsigma$ \emph{extended cubical sets;} we notate the category of
extended cubical sets as $\qsSet$.  We write $\square^n_\Sigma$ for
the representable presheaf $\qsigma({-},\sqparen{n})$.  To complete
the description of $\qsigma$ as a \textsc{prop} we need its symmetric
strict monoidal structure:
\begin{definition}
  Define $\sqparen{m}\oplus \sqparen{n} = \sqparen{m+n}$.  Suppose
  $X_i$ is a formal cubical $(m_i,n_i)$-product, $i=1,2$.  Define $X_1
  \oplus X_2$ as follows:
  \begin{enumerate}
  \item Replace each $x_j$ in $X_2$ by $x_{j+m_1}$ to form $X'_2$.
  \item Concatenate $X_1$ and $X'_2$.
  \end{enumerate}
  The symmetry $\sqparen{m}\oplus\sqparen{n} \to
  \sqparen{n}\oplus\sqparen{m}$ is the formal product
  \begin{equation*}
    (x_{m+1},x_{m+2},\dotsc,x_{m+n},x_1,x_2,\dotsc,x_m).
  \end{equation*}
\end{definition}
For example,
\begin{equation*}
  \big( (x_1\wedge x_2):\sqparen{2}\to\sqparen{1}\big) \oplus
  \big( (0,x_1) : \sqparen{1} \to \sqparen{2}\big) =
  (x_1\wedge x_2,0, x_3):\sqparen{3} \to \sqparen{3}.
\end{equation*}

\subsection{Generators and relations in $\qsigma$}
In order to describe skeletal filtrations on extended cubical sets, we
need a presentation of $\qsigma$.  The relations we list are similar
to those in \cite{grandismauri}.
\begin{definition}
  Suppose $n > 0$, $1\le i\le n+1$, and $\varepsilon = 0,1$.  Define
  maps $\delta^{i,\varepsilon}_n$ and $\sigma^i_n$ by the formal products
  \begin{align}
    \delta_n^{i,\varepsilon} &= (x_1,\dotsc,x_{i-1},\varepsilon,
    x_i, \dotsc, x_n) : \sqparen{n} \to \sqparen{n+1} \\
    \sigma_n^i  &= (x_1,\dotsc,x_{i-1}, x_{i+1}, \dotsc,x_{n+1}):
    \sqparen{n+1} \to \sqparen{n}. \notag
  \end{align}
  For $n \ge 1$ and $i \le n$, we define $\gamma^i_n$ to be
  \begin{equation*}
    \gamma^i_n = (x_1,\dotsc,x_{i-1}, x_i \wedge x_{i+1},
    x_{i+2},\dotsc,x_{n+1}) : \sqparen{n+1} \to \sqparen{n}.
  \end{equation*}
  Finally, for $p\in\Sigma_n$, we let
  \begin{equation*}
    \pi_p = (x_{p^{-1}(1)}, \dotsc, x_{p^{-1}(n)}) : \sqparen{n} \to \sqparen{n}.    
  \end{equation*}
\end{definition}

Note that $\QQ$ is isomorphic to the subcategory of $\qsigma$
generated by the coface and codegeneracy maps
$\delta^{i,\varepsilon}_n$ and $\sigma^i_n$.  The inclusion $\QQ \to
\qsigma$ is the strict monoidal functor we described above in terms of
forgetting structure.

\begin{proposition}\label{prop:cocubicalidentities}
  The codegeneracy and coface maps satisfy the following relations:
  \begin{equation}\label{eq:cocubicalidentities}
      \begin{split}
        \delta^{j,\eta} \delta^{i,\varepsilon} &=
        \delta^{i,\varepsilon}\delta^{j-1,\eta}\quad\text{if $i < j$} \\
        \sigma^j \delta^{i,\varepsilon} &=
        \begin{cases}
          \delta^{i,\varepsilon} \sigma^{j-1} & \text{if $i < j$} \\
          \id & \text{if $i = j$} \\
          \delta^{i-1,\varepsilon}\sigma^j & \text{if $i > j$}
        \end{cases} \\
        \sigma^j \sigma^i &= \sigma^i \sigma^{j+1}\quad\text{if $i \le j$}
      \end{split}
    \end{equation}
  The conjunction maps satisfy the following relations
  \cite{grandismauri}:
  \begin{align}\label{eq:extendedcocubicalidentities}
    \gamma^j \gamma^i &=
    \begin{cases}
      \gamma^i \gamma^{j+1} & \text{if $j > i$} \\
      \gamma^i \gamma^{i+1} & \text{if $j = i$}
    \end{cases} &
    \sigma^j \gamma^i &= 
    \begin{cases}
      \gamma^{i-1}\sigma^j & \text{if $j < i$} \\
      \sigma^i \sigma^i & \text{if $j = i$} \\
      \gamma^i \sigma^{j+1} & \text{if $j > i$}
    \end{cases} \\
    \gamma^j \delta^{i,\varepsilon} &=
    \begin{cases}
      \delta^{i-1,\varepsilon} \gamma^j & \text{if $j < i-1$} \\
      \delta^{i,0} \sigma^i & \text{if $j = i-1,i$ and $\varepsilon = 0$} \\
      \id & \text{if $j = i-1,i$ and $\varepsilon = 1$} \\
      \delta^{i,\varepsilon} \gamma^{j-1} & \text{if $j > i$} \notag
    \end{cases}
  \end{align}
\end{proposition}

\begin{definition}
  Let $\qsigma^+$ be the subcategory of $\qsigma$ generated by coface
  maps $\delta^{i,\varepsilon}_n$ and cosymmetry maps $\pi_p$; let
  $\qsigma^-$ be the subcategory of $\qsigma$ generated by
  codegeneracy maps $\sigma^i_n$, conjunctions $\gamma^i_n$, and
  cosymmetry maps $\pi_p$.
\end{definition}

\begin{proposition}\label{prop:qsigmafactorization}
  Every map $f$ in $\qsigma$ admits a unique factorization of the form
  \begin{equation}\label{eq:qsigmafactorization}
    f = \delta^{i_1,\varepsilon_1}\dotsb\delta^{i_n,\varepsilon_n}
    \gamma^{k_1} \dotsb \gamma^{k_r} \pi_p
    \sigma^{j_1}_\ell\dotsb \sigma^{j_m}
  \end{equation}
  with 
  \begin{align*}
    i_1 &> i_2 > \dotsb > i_n &
    j_1 &< j_2 < \dotsb < j_m &
    k_1 &< k_2 < \dotsb < k_r
  \end{align*}
  and $p\in\Sigma_\ell$.  If $f\in\Ar \qsigma^+$, then $r = m = 0$; if
  $f \in\Ar \qsigma^{-}$, then $n = 0$.  If $f\in\Ar \QQ$, then $r =
  0$ and $\pi_p = \id$.
\end{proposition}
We can read the decomposition \eqref{eq:qsigmafactorization} off of a
formal cubical $(a,b)$-product as follows: the indices $j_1,\dotsc,j_m$
correspond to the symbols $x_{j_1},\dotsc,x_{j_m}$ omitted from the
formal cubical product.  The concatenation of the remaining indices
determines $\pi_p$ uniquely; the list $k_1,\dotsc,k_r$ corresponds to
the positions in which a concatenation is performed, and the list
$i_1,\dotsc,i_n$ corresponds to the positions containing a $0$ or
$1$.  For example, the $(5,4)$-product
\begin{equation*}
  (x_3, 1, x_1 \wedge x_5 \wedge x_2, 0)
\end{equation*}
decomposes uniquely as
\begin{equation*}
  \delta^{5,0} \delta^{2, 1} \gamma^2 \gamma^3 \pi_{(1243)} \sigma^4.
\end{equation*}
This is analogous to the decompositions given by Grandis and Mauri
\cite[Theorem 8.3]{grandismauri}.  However, Grandis and Mauri's
extended cubical category $\mathbb{K}$ has an additional degeneracy
operation given by disjunction ${-}\vee{-}$ and some additional
relations.  More importantly, in $\mathbb{K}$, the operations $\wedge$
and $\vee$ are commutative.  As a result, the permutation $p$ in
factorizations of the form \eqref{eq:qsigmafactorization} in
$\mathbb{K}$ is uniquely determined up to multiplication of a possibly
nontrivial subgroup of $\Sigma_\ell$.  This has the upshot that the
vertices functor $\mathbb{K}(\sqparen{0},{-}) :
\mathbb{K}\to\Cat{Set}$ is faithful.  However, the analogous vertices
functor $\qsigma(\sqparen{0},{-}):\qsigma\to \Cat{Set}$ is not
faithful.  This is a marked departure from most cubical sites.

\begin{corollary}\label{cor:qsigmareedy}
  Suppose $f:\sqparen{m} \to \sqparen{n}$ is a map in $\qsigma$.  Then
  $f$ admits a factorization $f = \delta \sigma$ with $\delta \in
  \qsigma^{+}$ and $\sigma \in \qsigma^{-}$.  Given any other
  factorization $f = \delta' \sigma'$, there is a unique map $\pi$
  such that
  \begin{equation*}
    \xymatrix{ & \sqparen{r} \ar[dd]^(0.3)\pi \ar[dr]^\delta \\
      \sqparen{m} \ar'[r]^f[rr] \ar[ur]^\sigma \ar[dr]_{\sigma'} && \sqparen{n} \\
      & \sqparen{r} \ar[ur]_{\delta'}}
  \end{equation*}
  commutes; in fact $\pi$ is always a cosymmetry isomorphism.
\end{corollary}

\section{The structure of extended cubical sets}\label{section:qisreedy}
In this section, we'll present some machinery that allows us to
decompose cubical sets and symmetric cubical sets as colimits of their
skeleta.  We'll first need a workable definition of skeleton.  There
is a general theory due to Cisinski of skeletal decompositions
generalizing the classical theory for simplicial sets in
\cite{eilenbergzilber,gabrielzisman}---see \cite[Chap\^\i{}tre
8]{cisinskithesis}.  Moerdijk and Berger also discuss an apparatus for
skeletal decomposition in \cite{bergermoerdijkreedy}---the theory of
Eilenberg-Zilber categories---which we'll apply to $\QQ$ and
$\qsigma$.

\subsection{Eilenberg-Zilber categories and decompositions}
Suppose $X$ is a simplicial set.  Given any $n$-simplex $f:\Delta[n]
\to X$, we may take a factorization
\begin{equation}\label{eq:ezindelta}
  \vcenter{\xymatrix{
      \Delta[n] \ar[rr]^f \ar[rd]_s && X \\
      & \Delta[r] \ar[ur]_g
    }}
\end{equation}
so that $s:[n]\to[r]$ is an epimorphism (i.e., a degeneracy map) and
$r$ is minimal among all such factorizations.  Of course, the
minimality of $r$ implies that simplex $g$ is
\emph{nondegenerate}---it does not factor through another degeneracy.
One feature of the combinatorics of $\Delta$ is that this
factorization is unique: if $f = g's'$ is another factorization with
$g'$ nondegenerate and $s'$ a degeneracy map, then $g' = g$ and $s' =
s$.  This seemingly innocuous observation allows us to identify the
$m$-skeleton of $X$ (usually given as the counit of the left Kan
extension/restriction adjunction along $\Delta_{\le m} \to \Delta$) as
the subpresheaf of $X$ whose $n$-simplices are those $n$-simplices $f$
so that $r \le m$ in the Eilenberg-Zilber decomposition
\eqref{eq:ezindelta} of $f$.  A simple induction argument then shows
that the maps $\partial \Delta[n] \to \Delta[n]$, $n \ge 0$ comprise a
cellular model for $\sSet$.

These sort of arguments also work in $\qSet$ (as we'll see below), but
not in $\qsSet$ without some modification.  The identity map
$\square_\Sigma^n \to \square_\Sigma^n$ is nondegenerate, in the sense
that it does not factor through any non-invertible degeneracies, but
any symmetry $\pi$ yields a factorization $\pi^{-1}\pi$.  Also, the
maps $\iota_n:\partial \square^n_\Sigma \to \square^n_\Sigma$ do not
comprise a cellular model for $\qsSet$: there is no way to form, e.g.,
$\Sigma_2 \backslash \square^2_\Sigma$ with iterated cobase changes,
transfinite compositions, and retracts of the maps $\iota_n$.  As it
turns out, these are the only two complications that arise when we try
to apply Gabriel and Zisman's theory to $\qsigma$.  We need to replace
uniqueness with a properly categorical notion---contractible
groupoids---and we need to keep track of the action of
$\Aut(\square^n_\Sigma)$ on $\square^n_\Sigma$.  The appropriate
generalization of $\Delta$ is the notion of an \emph{Eilenberg-Zilber}
category, which we introduce below.  This generalization is due to
Berger, Moerdijk and Cisinski
\cite{bergermoerdijkreedy,cisinskithesis}.

\begin{definition}\label{def:absolutecolimit}
  Suppose $\mathscr{C}$ is a category and $\mathscr{I}$ a small
  category.  Suppose further that $X:\mathscr{I}\to\mathscr{C}$ is a
  diagram and $X\to c_Y$ is a cocone on $X$ (here $c_Y$ is the
  constant $\mathscr{I}$-diagram on $Y$).  We say $X\to
  c_Y$ is an \emph{absolute colimit} if $FX \to c_{FY}$ is a colimit for
  all functors $F:\mathscr{C}\to\mathscr{D}$.
\end{definition}
Split coequalizers are examples of absolute colimits \cite{maclane}.
In Definition \ref{def:absolutecolimit}, it is necessary and
sufficient to check that $[X] \to c_{[Y]}$ is a colimit; this is due
to Par\'e \cite{pareabsolutecolimits}.
\begin{definition}[{\cite[Definition 6.6]{bergermoerdijkreedy}}]\label{def:ezcategory}
  An \emph{Eilenberg-Zilber category} (briefly \emph{\textsc{ez} category}) is
  a small category $\mathscr{R}$ together with a degree function
  $\deg:\Ob\mathscr{R} \to \mathbf{Z}_{\ge 0}$ such that
  \begin{enumerate}
  \item[(EZ1)] Monomorphisms preserve the degree if and only if they
    are invertible; they raise the degree if and only if they are
    non-invertible.
  \item[(EZ2)] Every morphism factors as a split epimorphism followed by a
    monomorphism.
  \item[(EZ3)] Suppose
    \begin{equation*}
      \xymatrix{s_1 & r \ar[l]_{\sigma_1} \ar[r]^{\sigma_2} & s_2}
    \end{equation*}
    is a pair of split epimorphisms.  Then there is an absolute pushout square
    \begin{equation*}
      \xymatrix{r \ar[r]^{\sigma_2} \ar[d]_{\sigma_1} & s_2
        \ar[d]^{\tau_2} \\
        s_1 \ar[r]_{\tau_1} & t}
    \end{equation*}
    in $\mathscr{R}$ in which $\tau_1$ and $\tau_2$ are split
    epimorphisms.
  \end{enumerate}
\end{definition}
Suppose $\mathscr{R}$ is an \textsc{ez} category whose only
isomorphisms are identity maps.  The factorization provided by EZ2 is
then unique by EZ3.  Moreover, since the
section of a split epimorphism is monic, non-identity split
epimorphisms lower degree.  In this special case, $\mathscr{R}$ is an
example of a \emph{Reedy category:}
\begin{definition}
  Suppose $\mathscr{C}$ is a category and $\mathscr{D}$ a subcategory
  of $\mathscr{C}$; we say $\mathscr{D}$ is \emph{lluf} if
  $\Ob\mathscr{D} = \Ob\mathscr{C}$ (this terminology is due to Peter
  Freyd).  A \emph{Reedy category} \cite{dhks,hirschhorn,hovey} is a
  small category $\mathscr{R}$ together with a degree function
  $\deg:\Ob\mathscr{R} \to \mathbf{Z}_{\ge 0}$ and two lluf
  subcategories $\mathscr{R}^{+}$ and $\mathscr{R}^{-}$ so that
  \begin{enumerate}
  \item[(R1)] Non-identity morphisms in $\mathscr{R}^+$ raise the degree;
    non-identity morphisms  in $\mathscr{R}^-$ lower the degree.
  \item[(R2)] Every morphism $f\in\Ar\mathscr{R}$ factors uniquely as $f = gh$ with
  $g\in\Ar\mathscr{R}^{+}$ and $h\in\Ar\mathscr{R}^{-}$.
  \end{enumerate}
\end{definition}
Not all Reedy categories are \textsc{ez} though. %XXX counterexample
As expected, $\Delta$ is both.  The main result of this section is the
following:
\begin{proposition}\label{prop:qqisez}
  The categories $\QQ$ and $\qsigma$ are \textsc{ez} categories.
\end{proposition}
The proof of this, especially the verification of EZ3, is rather
technical and we postpone it to the end of the section.  Before we get
to it, we'll continue with a discussion of the properties of \textsc{ez}
categories.

\subsection{Skeleta, coskeleta, and cellular models}

\begin{definition}
  Let $\mathscr{R}$ be an \textsc{ez} category and suppose
  $X\in\widehat{\mathscr{R}}$.  We say a section $x\in X_r$ is
  \emph{degenerate} if there is a map $\sigma:s\to r$ in
  $\mathscr{R}^{-}$ and $y\in X_s$ so that $\sigma^\ast y = x$ and
  $\deg s < \deg r$.
\end{definition}

\begin{proposition}[{\cite[Proposition 6.7]{bergermoerdijkreedy}}]
  \label{prop:ezdecompositionsareunique}
  Let $\mathscr{R}$ be an \textsc{ez} category.
  \begin{enumerate}
  \item Suppose $X\in\widehat{\mathscr{R}}$.  Let $x\in X_r$,
    $r\in\Ob\mathscr{R}$.  The category of factorizations
    \begin{equation*}
      \xymatrix{[r] \ar[dr]_{\sigma_\ast} \ar[rr]^x && X
        \\
        & [s] \ar[ur]_y}
    \end{equation*}
    with $\sigma\in \mathscr{R}^{-}$ and $y$ nondegenerate is a
    contractible groupoid.
  \item If $f:r\to s$ is an arrow in $\mathscr{R}$, the category of
    factorizations
    \begin{equation*}
      \xymatrix{r \ar[dr]_{f^{-}} \ar[rr]^f && s
        \\
        & t \ar[ur]_{f^{+}}}
    \end{equation*}
    with $f^{-}$ a split epimorphism and $f^{+}$ a monomorphism is a
    contractible groupoid.
  \end{enumerate}
\end{proposition}
Following \cite{bergermoerdijkreedy,gabrielzisman}, we call any such
factorization an \emph{\textsc{ez} decomposition} of $x$.  The
Proposition implies in particular that \textsc{ez} decompositions
exist.

\begin{definition}
  Suppose $\mathscr{R}$ is an \textsc{ez} category and $n\ge -1$.  Write
  $\mathscr{R}_{\le n}$ for the full subcategory of $\mathscr{R}$ with
  objects those of degree at most $n$.  The inclusion
  $j_n:\mathscr{R}_{\le n} \to \mathscr{R}$ yields adjunctions
  \begin{equation}\label{eq:skeletaandcoskeleta}
    \vcenter{
      \xymatrix@C=4pc{\widehat{\mathscr{R}_{\le n}}
      \ar@<0.5ex>[r]^{(j_n)_!} &
      \widehat{\mathscr{R}} \ar@<0.5ex>[l]^{(j_n)^\ast}
      \ar@<0.5ex>[r]^{(j_n)^\ast} &
      \widehat{\mathscr{R}_{\le n}} \ar@<0.5ex>[l]^{(j_n)_\ast}
    }}
  \end{equation}
  given by left and right Kan extension.  We define the
  \emph{$n$-skeleton} and \emph{$n$-coskeleton} of $X\in\widehat{\mathscr{R}}$ to be
  \begin{align*}
    \sk_n X &= (j_n)_!(j_n)^\ast X &&\text{and} &
    \ck_n X &= (j_n)_\ast(j_n)^\ast X
  \end{align*}
  respectively.  The counit and unit of the adjunctions in
  \eqref{eq:skeletaandcoskeleta} yield natural maps
  \begin{equation*}
    \xymatrix{\sk_n X \ar[r] & X \ar[r] & \ck_n X.}
  \end{equation*}
  We say $X$ is \emph{$n$-skeletal} if $\sk_n X\to X$ is an
  isomorphism and \emph{$n$-coskeletal} if $X \to \ck_n X$ is an
  isomorphism.
\end{definition}

In a precise sense, the $n$-skeleton of $X \in \widehat{\mathscr{R}}$,
$\mathscr{R}$ an \textsc{ez} category, is the subpresheaf generated by the
non-degenerate sections of $X$ of degree at most $n$.
\begin{proposition}[\cite{bergermoerdijkreedy}]\label{prop:skiswelldefined}
  Suppose $\mathscr{R}$ is an \textsc{ez} category and
  $X\in\widehat{\mathscr{R}}$.  The map $\sk_n X \to X$ is a
  monomorphism; its image in $X_r$, $r\in\Ob\mathscr{R}$ is the set
  of sections
  \begin{equation*}
    \big\{ f\in X_r \bigm| \text{$f$ has a factorization
      $[r] \to [s] \to X$ with $\deg s \le n$}\big\}.
  \end{equation*}
\end{proposition}

\begin{definition}
  Suppose $\mathscr{R}$ is an \textsc{ez} category and $r\in \Ob\mathscr{R}$.
  We define the \emph{boundary} $\partial[r]$ of $[r]$ to be the
  $(n-1)$-skeleton of $[r]$, where $\deg r = n$.
\end{definition}

\begin{proposition}\label{prop:cellularmodelforez}
  Suppose $\mathscr{R}$ is an \textsc{ez} category whose only isomorphisms are
  identity maps.  Suppose $X\in\mathscr{R}$ and $n\ge 0$.  Let $S$ be
  the set of maps $f: [r] \to X$ with $\deg r = n$ and $f$
  nondegenerate.  The square
  \begin{equation}\label{eq:cellularmodelforez}
    \vcenter{
      \xymatrix{\coprod_{f:[r]\to X\in S} \partial [r] \ar[r] \ar[d] &
        \sk_{n-1} X \ar[d] \\
        \coprod_{f:[r]\to X\in S} [r] \ar[r] & \sk_n X
      }}
    \end{equation}
  is a pushout.
\end{proposition}
\begin{proof}
  This proof is a straightforward generalization of \cite[\S
  II.3.8]{gabrielzisman}.  Since every object in
  \eqref{eq:cellularmodelforez} is $n$-skeletal, it is sufficient to
  check that the restriction of \eqref{eq:cellularmodelforez} to
  $\mathscr{R}_{\le n}$ is a pushout square.  In a presheaf topos,
  pushouts are computed pointwise, so it is sufficient to prove that
  the square \eqref{eq:cellularmodelforez} is a pushout after
  evaluation at $s$ for all $s\in\Ob\mathscr{R}_{\le n}$.  If $\deg s
  < n$ and $\deg r = n$, the maps
  \begin{align*}
    \big(\partial [r]\big)_s &\to [r]_s &&\text{and} &
    (\sk_{n-1} X)_s &\to (\sk_n X)_s,
  \end{align*}
  are isomorphisms.  Thus we are reduced to checking that
  \begin{equation}\label{eq:cellularmodelforezpushout}
    \vcenter{
      \xymatrix{\coprod_{f:[r]\to X\in S} \big(\partial [r]\big)_s \ar[r] \ar[d] &
        (\sk_{n-1} X)_s \ar[d] \\
        \coprod_{f:[r]\to X\in S} [r]_s \ar[r] & (\sk_n X)_s
      }}
  \end{equation}
  is a pushout when $\deg s = n$.
  
  Suppose $\deg s = n$.  The complement of $(\sk_{n-1} X)_s$ in
  $(\sk_n X)_s$ is the set of all nondegenrate $s$-simplices $[s] \to
  X$.  Since $\mathscr{R}$ has no nontrivial isomorphisms, if $r \ne
  s$ has degree $n$, each map $s \to r$ factors through an object of
  lower degree, so 
  \begin{equation*}
    \big(\partial [r]\big)_s \to [r]_s
  \end{equation*}
  is an isomorphism.  On the other hand, the complement of the image
  of
  \begin{equation*}
    \big(\partial [s]\big)_s \to [s]_s
  \end{equation*}
  is the identity map $s\to s$, so the complement of the image of
  \begin{equation*}
    \coprod_{f:[r]\to X\in S} \big(\partial [r]\big)_s \to
    \coprod_{f:[r]\to X\in S} [r]_s
  \end{equation*}
  is the set of nondegenerate $s$-simplices $[s]\to X$.  Hence
  \eqref{eq:cellularmodelforezpushout} is a pushout.
\end{proof}

We can reinterpret Proposition \ref{prop:cellularmodelforez} as a
statement about saturated classes of maps.  We first introduce the following
definition, using Cisinski's terminology \cite{cisinskithesis}:
\begin{definition}
  Suppose $\mathscr{R}$ is a small category.  We say that a set of
  arrows $S \subseteq \Ar \widehat{\mathscr{R}}$ is a \emph{cellular
    model} for $\widehat{\mathscr{R}}$ if $\Cell S = \mathsf{mono}$.
  Here, $\Cell S$ is the closure of the set $S$ under transfinite
  composition, cobase change, coproduct, and retract. %XXX CHECK ME
\end{definition}
Any topos has a cellular model \cite{cisinskitopos,beke1}.  In the
case of a presheaf topos, all inclusions of subobjects of (regular)
quotients of representables form a cellular model.  In \textsc{ez} categories,
we have the expected simplification:
\begin{corollary}\label{cor:cellularmodelforez}
  Under the assumptions of Proposition \ref{prop:cellularmodelforez},
  the arrows $\partial [r]\to [r]$, $r\in\Ob\mathscr{R}$, comprise a
  cellular model for $\widehat{\mathscr{R}}$.
\end{corollary}
This corollary may seem slightly weaker than Proposition
\ref{prop:cellularmodelforez} because it does not say anything about
the dimension of the attaching maps (bringing to mind the distinction
between cellular and \textsc{cw} complexes in $\Cat{Top}$).  In
$\widehat{\mathscr{R}}$ however, every map $\partial
[r] \to X$ automatically factors through $\sk_{\deg r - 1} X\to X$.
\begin{proof}[Proof of Corollary \ref{cor:cellularmodelforez}]
  Let $C$ temporarily denote the class of arrows
  \begin{equation*}
    \Cell \{ \partial [r] \to [r] \mid r\in\Ob\mathscr{R} \}.
  \end{equation*}
  Since $\widehat{\mathscr{R}}$ is a topos, $C\subseteq
  \mathsf{mono}$.  Recall that $\sk_{-1} A = \sk_{-1} B = \emptyset$.
  Suppose $f:A\to B$ is a monomorphism in $\widehat{\mathscr{R}}$.
  Let $\sk_n f$ be the pushout $\sk_n B \amalg_{\sk_n A} A$ and let
  $p_n:\sk_n f\to B$ be the corner map.  Note that the square
  \begin{equation*}
    \xymatrix{\sk_n A \ar[r]^{\sk_n f} \ar[d] & \sk_n B \ar[d] \\
      A \ar[r]_f & B}
  \end{equation*}
  is a pullback.  Since $\widehat{\mathscr{R}}$ is a topos, $\sk_n f$
  is the effective union of $\sk_n B$ and $A$ in $\sk_n f$ and $p_n$ is a
  monomorphism.  The square
  \begin{equation*}
    \xymatrix{\sk_n B \ar[r] \ar[d] & \sk_{n+1} B \ar[d] \\
      \sk_n f \ar[r] & \sk_{n+1} f}
  \end{equation*}
  is a pushout, so $\sk_n f \to \sk_{n+1} f$ is in $C$.
  Now $\colim_n \sk_n f \to B$ is an isomorphism.  Since $\sk_{-1} f =
  A$, we've realized $f$ as a transfinite composition of maps in $C$,
  so $f\in C$.  Hence $C = \mathsf{mono}$.
\end{proof}

Note that Proposition \ref{prop:cellularmodelforez} is false if we
allow objects in $\mathscr{R}$ to have nontrivial automorphisms.  An
easy example is the one-object category associated to a group $G$.
This is an \textsc{ez} category with $\deg\ast = 0$.  An object
$X\in\widehat{G}$ is a right $G$-set; were Proposition
\ref{prop:cellularmodelforez} true, it would imply that all
$X\in\widehat{G}$ are free as $G$-sets.  We'll now prove a generalization
of Proposition \ref{prop:cellularmodelforez} for categories
$\mathscr{R}$ containing nontrivial isomorphisms.

\begin{definition}
  Suppose $\mathscr{R}$ is an \textsc{ez} category, $X\in
  \widehat{\mathscr{R}}$, and $f:[r] \to X$ is a nondegenerate
  $r$-simplex of $X$.  Note that $\Aut(r)$ acts on $X_r$ on the right.
  The \emph{isotropy of $f$,} denoted $\Stab(f)$, is the stabilizer of
  $f\in X_r$ in $\Aut(r)$, i.e., the subgroup of $g\in \Aut(r)$ with
  $g^\ast f = f$.
\end{definition}

In the following, note that the left action of $\Aut(r)$ on $[r]$
restricts to an action on $\partial [r]$.  If $H \leqslant \Aut(r)$,
then $H\backslash(\partial [r]) \to H\backslash [r]$ is a monomorphism
(here $H\backslash X$ denotes the $H$-orbits of $X$).

\begin{proposition}\label{prop:cellularmodelwithnontrivialauts}
  Suppose $\mathscr{R}$ is a skeletal \textsc{ez} category, i.e., two objects
  are isomorphic if and only if they are equal.  Let $n\ge 0$ and let
  $S$ be a set of isomorphism classes of $f:[r] \to X$ with $\deg r =
  n$ and $f$ nondegenerate.  Then the square
  \begin{equation*}
    \xymatrix{\coprod_{f:[r] \to X \in S} \Stab f \backslash \partial
      [r] \ar[r] \ar[d] & \sk_{n-1} X \ar[d] \\
      \coprod_{f:[r]\to X \in S} \Stab f \backslash [r] \ar[r] & \sk_n X}
  \end{equation*}
  is a pushout.
\end{proposition}
\begin{proof}
  Note that since $\sk_n$ is cocontinuous, if $Y\in\mathscr{R}$ is
  $n$-skeletal and a group $G$ acts on $Y$, $G\backslash Y$ is
  $n$-skeletal as well.  Just as in the proof of Proposition
  \ref{prop:cellularmodelforez}, it is sufficient to check that
  \begin{equation}\label{eq:cellularmodelnontrivautpushout}
    \vcenter{
      \xymatrix{\coprod_{f:[r]\to X\in S} \big(\Stab f \backslash \partial [r]\big)_s \ar[r] \ar[d] &
        (\sk_{n-1} X)_s \ar[d] \\
        \coprod_{f:[r]\to X\in S} \big(\Stab f \backslash [r]\big)_s \ar[r] & (\sk_n
        X)_s
      }}
  \end{equation}
  is a pushout when $\deg s = n$.  The complement of $(\sk_{n-1}
  X)_s$ in $(\sk_n X)_s$ is the right $\Aut(s)$-set of all
  nondegenerate $s$-simplices $f:[s] \to X$.
  
  Now we have two possibilities.  Suppose $r\ne s$ has degree $n$.  If
  $f:[r]\to X$ is a nondegenerate $r$-simplex and $r\ne s$,
  \begin{equation*}
    \big(\Stab f \backslash \partial [r]\big)_s \to \big(\Stab f
    \backslash [r]\big)_s
  \end{equation*}
  is an isomorphism: any map $s\to r$ must factor through an object of
  lower degree since a degree-preserving map $s\to r$ is necessarily
  an isomorphism (recall that we have assumed $\mathscr{R}$ is
  skeltal).  On the other hand, if $H \leqslant \Aut(s)$, the
  complement of the image of
  \begin{equation*}
    \big(H\backslash \partial [s]\big)_s \to \big(H\backslash [s]\big)_s
  \end{equation*}
  is the right $\Aut(s)$-set $H\backslash \Aut(s)$.  Thus the
  complement of the image of
  \begin{equation*}
    \coprod_{f:[r]\to X\in S} \big(\Stab f \backslash \partial [r]\big)_s \to
    \coprod_{f:[r]\to X\in S} \big(\Stab f \backslash [r]\big)_s
  \end{equation*}
  is the right $\Aut(s)$-set
  \begin{equation*}
    \coprod_{g:[s] \to X\in S} \Stab(g) \backslash \Aut(s)
  \end{equation*}
  This decomposition maps isomorphically onto $(\sk_n X)_s \setminus
  (\sk_{n-1} X)_s$ via the map sending the coset $\Stab(g)$ to $g$.
  Hence \eqref{eq:cellularmodelnontrivautpushout} is a pushout.
\end{proof}

\begin{corollary}
  Under the assumptions of Proposition
  \ref{prop:cellularmodelwithnontrivialauts}, the set
  \begin{equation*}
    \big \{ H\backslash \partial [r] \to H\backslash [r] \bigm| \text{$r\in\Ob\mathscr{R}$
    and $H \leqslant \Aut r$} \big\}
  \end{equation*}
  is a cellular model for $\widehat{\mathscr{R}}$.
\end{corollary}

As we'll see below, $\QQ$ and $\qsigma$ are \textsc{ez} categories, so we
obtain the following cellular models.  We write $\square^n$ and
$\square^n_\Sigma$ for the representables $\QQ({-}, \sqparen{n})$ and
$\qsigma({-}, \sqparen{n})$ respectively.
\begin{proposition}\label{prop:cellularmodelforqsigma}
  The sets
  \begin{gather}
    I = \big\{ \partial \square^n \to \square^n \bigm| n \ge 0 \big\} \\
    I_\Sigma = \big\{ H\backslash \partial \square_\Sigma^n \to
    H\backslash \square^n_\Sigma \bigm| \text{$n \ge 0$ and $H\leqslant
    \Sigma_n$} \big\}\notag
  \end{gather}
  are cellular models for $\qSet$ and $\qsSet$,
  respectively.
\end{proposition}

\subsection{Comparing skeletal filtrations}
In this section we'll prove a base-change theorem that allows us to
compare skeleta of cubical and extended cubical sets.  We begin with a
slightly modified definition from \cite[Chap\^\i{}tre 8]{cisinskithesis}:
\begin{definition}
  Suppose $i:\mathscr{R}\to\mathscr{S}$ is a functor.  We say that $i$
  is a \emph{thickening} if
  \begin{enumerate}
  \item $i$ is an isomorphism on objects.
  \item For all $r,r'\in\Ob\mathscr{R}$, the map
    \begin{equation*}
      \Aut_\mathscr{S}(ir) \times \mathscr{R}(r,r') \to
      \mathscr{S}(ir, ir')
    \end{equation*}
    is a bijection of sets.
  \end{enumerate}
\end{definition}
Crossed $\Delta$-modules, and more generally crossed
$\mathscr{R}$-modules for a Reedy category $\mathscr{R}$, are examples
of thickenings \cite{bergermoerdijkreedy,fiedorowiczloday}.  Note that
$\QQ\to\qsigma$ is not a thickening, but
$\QQ^{+}\to\qsigma^{+}$ is a thickening by Proposition
\ref{prop:qsigmafactorization}.  We start with a simple
observation:
\begin{lemma}\label{lemma:thickeningfactorization}
  Suppose $i:\mathscr{R}\to\mathscr{S}$ is a thickening and 
  \begin{equation*}
    \xymatrix{& ir_2 \ar[rd]^{if} \\
      ir_1 \ar[ur]^{\sigma} \ar[rr]_{ig} && ir_3}
  \end{equation*}
  is a diagram in which $\sigma$ is an arrow in $\mathscr{S}$.  Then
  $\sigma$ is in the image of $i$ and the triangle may be lifted to
  one in $\mathscr{R}$.
\end{lemma}
\begin{proof}
  Since $i$ is a thickening, there is a (unique) factorization of
  $\sigma$ as a composition $(ih)\circ \tau$, where
  $\tau\in\Aut_\mathscr{S}(ir_1)$ and $h$ is a map $r_1\to r_2$.  Then
  $i(gh)\circ \tau = if$, so by the uniqueness of factorizations of
  this form, $\tau = \id_{ir_1}$ and hence $\sigma = ih$.
\end{proof}

\begin{proposition}\label{prop:basechangeforskeleta}
  Suppose $i:\mathscr{R}\to\mathscr{S}$ is a functor between
  \textsc{ez} categories $\mathscr{R}$ and $\mathscr{S}$ so that
  $i$ preserves degree.  Then $i$ preserves monomorphisms; suppose
  that moreover, the resulting functor $i^{+}:\mathscr{R}^{+} \to
  \mathscr{S}^{+}$ is a thickening.  The natural base-change
  transformation $i_!  j^\ast \to j^\ast i_!$ of functors
  $\widehat{\mathscr{R}} \to \widehat{\mathscr{S}_{\le n}}$ induced by
  the square of functors
  \begin{equation*}
    \xymatrix{\mathscr{R}_{\le n} \ar[r]^j \ar[d]_i & \mathscr{R}
      \ar[d]^i \\
      \mathscr{S}_{\le n} \ar[r]_j & \mathscr{S}}
  \end{equation*}
  is a natural isomorphism.
\end{proposition}
\begin{proof}
  The functors $i_!$ and $j^\ast$ preserve colimits, so it is
  sufficient to check that $i_! j^\ast \to j^\ast i_!$ is an
  isomorphism on all representables $[r] \in \widehat{\mathscr{R}}$,
  $r\in\Ob\mathscr{R}$.  Suppose $r\in\Ob\mathscr{R}$ and
  $s\in\Ob\mathscr{S}_{\le n}$.  Note that $j^\ast [r] \cong \colim_{r'}
  [r']$ where the colimit is taken over all $r'$ with $\deg r' < n$,
  so $i_! j^\ast [r] \cong \colim_{r'} [ir']$.
  As a result, on the level of sets, the map
  $\varphi:\big(i_! j^\ast [r]\big)_s \to \big(j^\ast i_! [r]\big)_s$ is
  given by
  \begin{equation*}
    \colim_{\substack{r' \to r\\\deg r' \le n}} \big\{ s \to
    ir' \in \Ar\mathscr{S} \big\} \to \{ s \to ir \in \Ar \mathscr{S} \}
  \end{equation*}
  Now suppose $g:s\to ir$ is a map in $\mathscr{S}$.  Since
  $\mathscr{S}$ is an \textsc{ez} category and
  $i^{+}$ is a thickening, there is a factorization
  \begin{equation*}
    \xymatrix{ & ir' \ar[rd]^{ig^{+}} \\
      s \ar[ur]^{g^{-}} \ar[rr]_g && ir}
  \end{equation*}
  in which $g^{-}$ is a split epimorphism in $\mathscr{S}$ and $g^{+}$
  is a monomorphism in $\mathscr{R}$.  Since $\deg s \le n$, $\deg r'
  \le n$ as well.  Hence $\varphi$ is a surjection.  We can assume,
  moreover, that the degree of $r'$ is minimal among all such
  factorizations (in fact, there is only one possible degree).
  
  Now suppose we have maps $h:s\to ir_1$ in $\mathscr{S}$ and
  $\ell:r_1\to r$ in $\mathscr{R}$ so that $\deg r_1 \le n$ and $i\ell
  \circ h= g$.  We must show that the pair $(\ell, h)$ is identified
  with $(g^{+}, g^{-})$ in the colimit in the source of $\varphi$.
  By repeated factorizations, we can produce a diagram
  \begin{equation*}
    \xymatrix{s \ar[r]^h \ar[rd]_{h^{-}} & ir_1 \ar[rr]^{i\ell}
      \ar[rd]^{i\ell^{-}} && ir \\
      & ir_3 \ar[r]_{ih^{+}} & ir_2
      \ar[ur]_{i\ell^{+}}}
  \end{equation*}
  in which $h^{-}$ is a split epimorphism in $\mathscr{S}$, $\ell^{-}$
  is a split epimorphism in $\mathscr{R}$, and both $\ell^{+}$ and $h^{+}$ are
  monomorphisms in $\mathscr{R}$.  The pairs
  \begin{equation*}
    (\ell,h) \qquad (\ell^{+}, h\circ i\ell^{-}) \qquad
    (\ell^{+} h^{+}, h^{-})
  \end{equation*}
  are identified in the colimit in the source of $\varphi$.  (Note
  that $\deg r_3 \le \deg r_1$.)  Without loss of generality, then, we
  can assume that $h$ is a split epimorphism and $\ell$ is a
  monomorphism.  But split epi-monic factorizations in $\mathscr{S}$
  are essentially unique (Proposition
  \ref{prop:ezdecompositionsareunique}), so there is an isomorphism
  $\sigma$ making
  \begin{equation*}
    \xymatrix{ & ir' \ar[dd]^(0.3)\sigma \ar[dr]^{ig^{+}} \\
      s \ar'[r]^g[rr] \ar[ur]^{g^{-}} \ar[dr]_{h} && ir \\
      & ir_1 \ar[ur]_{i\ell}}
  \end{equation*}
  commute.  Since $i^{+}$ is a thickening, the map $\sigma$ must be
  in the image of $i$ (Lemma \ref{lemma:thickeningfactorization}), so
  $(\ell, h)$ and $(g^{+}, g^{-})$ are identified in the colimit in
  the source of $\varphi$.  Hence $\varphi$ is a bijection of sets.
\end{proof}

\begin{corollary}\label{cor:skeletoncomparison}
  Suppose $i:\mathscr{R}\to\mathscr{S}$ is a functor between
  \textsc{ez} categories satisfying the assumptions of
  Proposition \ref{prop:basechangeforskeleta}.  If
  $X\in\widehat{\mathscr{R}}$, then there is a natural isomorphism
  $\sk_n i_! X \to i_! \sk_n X$.
\end{corollary}
\begin{proof}
  With the notation of Proposition \ref{prop:basechangeforskeleta},
  there is a natural isomorphism $i_! j^\ast X \to j^\ast i_! X$.  Now
  apply the functor $j_!$; we obtain a natural isomorphism $j_! i_!
  j^\ast X \to j_! j^\ast i_! X$.  Since $j_! i_! \cong i_! j_!$, we
  obtain a natural isomorphism $i_! j_! j^\ast X \to j_! j^\ast i_!
  X$.
\end{proof}

\subsection{The cubical sites}
In the remainder of this section, we'll prove Proposition
\ref{prop:qqisez}.  We begin with the definition $\deg \sqparen{n} =
n$.  Axioms EZ1 and EZ2 are routine; the difficult axiom to verify
will be EZ3.
\begin{lemma}
  All epimorphisms of $\QQ$ and $\qsigma$ are split.  The epimorphisms
  of $\QQ$ (respectively $\qsigma$) correspond to the arrows of
  $\QQ^{-}$ (respectively $\qsigma^{-}$).
\end{lemma}
\begin{proof}
  Both the degeneracies $\sigma^i$ and $\gamma^i$ have sections, so
  they are categorical epimorphisms.  Since the cosymmetry maps
  $\pi_p$ are isomorphisms, we may conclude that the arrows of
  $\qsigma^{-}$ and $\QQ^{-}$ are split epimorphisms in $\qsigma$
  and $\QQ$, respectively.

  Suppose $f$ is an arrow in $\qsigma$.  We may factor $f$ as 
  \begin{equation*}
    f = \delta^{i_1,\varepsilon_1}\dotsb\delta^{i_n,\varepsilon_n} s
  \end{equation*}
  with $s\in\Ar \qsigma^{-}$.  Suppose $n > 0$.  Then
  \begin{equation*}
    f = \delta^{i_1,\varepsilon_1} \sigma^{i_1} f
  \end{equation*}
  by the relations in Proposition \ref{prop:cocubicalidentities}.
  However, $\delta^{i_1,\varepsilon_1} \sigma^{i_1} \ne \id$, so $f$
  is not an epimorphism.  Hence the epimorphisms of $\qsigma$ are
  precisely the maps of $\qsigma^{-}$, which are all split.  The proof
  for $\QQ$ is identical.
\end{proof}

\begin{definition}\label{def:splitpushout}
  Suppose
  \begin{equation}\label{eq:absolutepushout}
    \vcenter{
      \xymatrix{A \ar[r]^{a_1} \ar[d]_{a_2} & B \ar[d]^{p_2} \\
        C \ar[r]_{p_1} & P
      }}
  \end{equation}
  is a commutative square in a category $\mathscr{C}$.  If there exist
  maps
  \begin{align*}
    d_0 &: P \to B &
    d_1 &: C \to A &
    d'_1 &: B \to A &
    d'_2 &: B \to A
  \end{align*}
  so that
  \begin{align*}
    d_0 p_1 &= a_1 d_1 &
    d_0 p_2 &= a_1 d'_1 \\
    a_2 d_1 &= \id_C &
    a_2 d'_1 &= a_2 d'_2 \\
    p_2 d_0 &= \id_P &
    a_1 d'_2 &= \id_B
  \end{align*}
  then we call \eqref{eq:absolutepushout} a \emph{split pushout}.
\end{definition}

\begin{lemma}\label{lemma:absolutepushout}
  Split pushouts are absolute pushouts.
\end{lemma}
\begin{proof}
  This is an example of a general criterion by Par\'e, who classifies
  all absolute pushouts in \cite[Proposition
  5.5]{pareabsolutecolimits} (the cited paper also classifies all
  absolute colimits in general).  It is sufficient to check that an
  split pushout of the shape \eqref{eq:absolutepushout} is a pushout
  square, since split pushouts are manifestly preserved by all
  functors.  Suppose $f:B\to X$ and $g:C\to X$ are given so that $fa_1
  = ga_2$.  Define $h:P\to X$ to be $h = f d_0$.  Then
  \begin{equation*}
    h p_1 = f d_0 p_1 = f a_1 d_1 = ga_2 d_1 = g
  \end{equation*}
  and
  \begin{equation*}
    h p_2 = f d_0 p_2 = f a_1 d'_1 = ga_2 d'_1 = ga_2 d'_2 = fa_1 d'_2
    = f.
  \end{equation*}
  Since $p_2$ is a split epimorphism, $h$ is the unique map making
  \begin{equation*}
    \xymatrix{A \ar[r]^{a_1} \ar[d]_{a_2} & B \ar[d]^{p_2} \ar@/^/[ddr]^f \\
      C \ar[r]_{p_1} \ar@/_/[rrd]_{g} & P \ar[rd]^h \\
      && X}
  \end{equation*}
  commute.
\end{proof}

In all the split pushouts we compute below, we will always set $d'_1 =
d'_2 d_0 p_2$.  This reduces the relations that we need to verify to
the following five:
\begin{align*}
  a_2 d_1 &= \id_C &
  d_0 p_1 &= a_1 d_1 \\
  p_2 d_0 &= \id_P &
  a_2 d'_2 d_0 p_2 &= a_2 d'_2 \\
  a_1 d'_2 &= \id_B
\end{align*}

\begin{lemma}\label{lemma:pushoutofsigmasisabsolute}
  The diagram
  \begin{equation}\label{eq:qqsplitepipushout}
    \vcenter{
      \xymatrix{\sqparen{n-1} & \sqparen{n} \ar[l]_{\sigma^i} \ar[r]^{\sigma^j} &
        \sqparen{n-1}
      }}
  \end{equation}
  has an absolute pushout
  \begin{equation*}
    \xymatrix{\sqparen{n} \ar[r]^{\sigma^i} \ar[d]_{\sigma^j} & \sqparen{n-1} \ar[d]^{\tau_2} \\
      \sqparen{n-1} \ar[r]_{\tau_1} & \sqparen{\ell} }
  \end{equation*}
  in $\QQ$ with both $\tau_1$ and $\tau_2$ in $\QQ^{-}$ and $n-2 \le
  \ell \le n-1$.
\end{lemma}

\begin{proposition}\label{prop:gridinduction}
  The category $\QQ$ satisfies axiom EZ3 of Definition \ref{def:ezcategory}.
\end{proposition}

\begin{corollary}
  The diagram
  \begin{equation*}
    \vcenter{
      \xymatrix{\sqparen{n-1} & \sqparen{n} \ar[l]_(0.35){\sigma^i}
      \ar[r]^(0.45){\sigma^j} &
        \sqparen{n-1}
      }}
  \end{equation*}
  has an absolute pushout
  \begin{equation*}
    \xymatrix{\sqparen{n} \ar[r]^{\sigma^i} \ar[d]_{\sigma^j} &
    \sqparen{n-1} \ar[d]^{\tau_2} \\
      \sqparen{n-1} \ar[r]_{\tau_1} & \sqparen{\ell} }
  \end{equation*}
  in $\qsigma$ with both $\tau_1$ and $\tau_2$ in $\qsigma^{-}$ and $n-2 \le
  \ell \le n-1$.
\end{corollary}
\begin{proof}
  The functor $i:\QQ\to\qsigma$ preserves absolute pushouts.
\end{proof}

\begin{proof}[Proof of Lemma \ref{lemma:pushoutofsigmasisabsolute}]
  If $i = j$, then the pushout of
  \eqref{eq:qqsplitepipushout} is $\sqparen{n-1}$ and is preserved by any
  functor $\QQ \to \mathscr{C}$.  Suppose $i < j$.  Then the square
  \begin{equation}\label{eq:qqsplitepipushout3}
    \vcenter{
      \xymatrix{\sqparen{n} \ar[r]^{\sigma^i} \ar[d]_{\sigma^j} & \sqparen{n-1}
        \ar[d]^{\sigma^{j-1}} \\
        \sqparen{n-1} \ar[r]_{\sigma^i} & \sqparen{n-2}
      }}
  \end{equation}
  is a split pushout in $\QQ$.  Using the notation of Definition
  \ref{def:splitpushout},
  we define sections
  \begin{align*}
    d_0 &= \delta^{j-1,0}_{n-2} &
    d_1 &= \delta^{j,0}_{n-1} &
    d'_2 &= \delta^{i,0}_{n-1}.
  \end{align*}
  Observe that these maps are well-defined, since $1 \le i < j \le n$.
  Using the cubical relations in Proposition
  \ref{prop:cocubicalidentities}, we verify the five relations:
  \begin{align*}
    \sigma^j \delta^{j,0} &= \id_{\sqparen{n-1}} & \text{($a_2d_1 = \id_C$)}\\
    \sigma^{j-1} \delta^{j-1,0} &= \id_{\sqparen{n-2}} & \text{($p_2d_0 = \id_P$)}\\
    \sigma^i \delta^{i,0} &= \id_{\sqparen{n-1}} & \text{($a_1d'_2 = \id_B$)} \\
    \delta^{j-1,0} \sigma^i &= \sigma^i \delta^{j,0} & \text{($d_0p_1 =
      a_1d_1$)} \\
    \sigma^j \delta^{i,0} \delta^{j-1,0} \sigma^{j-1} &= 
    \delta^{i,0} \sigma^{j-1} \delta^{j-1,0} \sigma^{j-1} =
    \delta^{i,0} \sigma^{j-1} = \sigma^j \delta^{i,0} & \text{($a_2 d'_2 d_0 p_2
      = a_2 d'_2$)}
  \end{align*}
  Hence \eqref{eq:qqsplitepipushout3} is an absolute pushout.
\end{proof}

\begin{lemma}
  Suppose $n \ge 2$.  The diagram 
  \begin{equation}\label{eq:pushoutofgammas}
    \vcenter{
      \xymatrix{\sqparen{n-1} & \sqparen{n} \ar[l]_(0.35){\gamma^i}
        \ar[r]^(0.45){\gamma^j} & \sqparen{n-1}
      }}
  \end{equation}
  has an absolute pushout $\sqparen{\ell}$ in $\qsigma$ with maps $\sqparen{n-1}
  \to \sqparen{\ell}$ in $\qsigma^{-}$ and $n-2 \le \ell \le n-1$.
\end{lemma}
\begin{proof}
  The proof of this lemma is similar to that of Lemma
  \ref{lemma:pushoutofsigmasisabsolute}.  Without loss of generality,
  we may assume that $j\ge i$.  We then have three special cases:
  \begin{enumerate}
  \item ($j = i$) The pushout of
    \eqref{eq:pushoutofgammas} is $\sqparen{n-1}$ and is absolute.
  \item ($j = i+1$) The square
    \begin{equation*}
      \xymatrix{\sqparen{n} \ar[r]^{\gamma^i} \ar[d]_{\gamma^{i+1}} &
        \sqparen{n-1} \ar[d]^{\gamma^i} \\
        \sqparen{n-1} \ar[r]_{\gamma^i} & \sqparen{n-2}}
    \end{equation*}
    is a a split pushout: define sections
    \begin{align*}
      d_0 &= \delta^{i+1,1} & d_1 &= \delta^{i+2,1} & d'_2 &= \delta^{i,1}.
    \end{align*}
    Note that $d_1:\sqparen{n-1}\to\sqparen{n}$ is well-defined, since
    $i+2$ is at most $n$.  Now we verify the five relations:
    \begin{align*}
      \gamma^{i+1}\delta^{i+2,1} &= \id_{\sqparen{n-1}} & \text{($a_2d_1 = \id_C$)}\\
      \gamma^i \delta^{i+1,1} &= \id_{\sqparen{n-2}} & \text{($p_2d_0 = \id_P$)}\\
      \gamma^i \delta^{i,1} &= \id_{\sqparen{n-1}} & \text{($a_1d'_2 = \id_B$)} \\
      \delta^{i+1,1} \gamma^i &= \gamma^i \delta^{i+2,1} & \text{($d_0p_1 =
        a_1d_1$)} \\
      \gamma^{i+1} \delta^{i,1} \delta^{i+1,1} \gamma^i &= 
      \gamma^{i+1} \delta^{i+2,1} \delta^{i,1} \gamma^i =
      \delta^{i,1} \gamma^i = \gamma^{i+1} \delta^{i,1}
      & \text{($a_2 d'_2 d_0 p_2
      = a_2 d'_2$)}
    \end{align*}
  \item ($j > i+1$) The square
    \begin{equation*}
      \xymatrix{\sqparen{n} \ar[r]^{\gamma^i} \ar[d]_{\gamma^j} &
        \sqparen{n-1} \ar[d]^{\gamma^{j-1}} \\
        \sqparen{n-1} \ar[r]_{\gamma^i} & \sqparen{n-2}}
    \end{equation*}
    is a split pushout: we define sections
    \begin{align*}
      d_0 &= \delta^{j-1,1} &
      d_1 &= \delta^{j,1} &
      d'_2 &= \delta^{i,1}.
    \end{align*}
    The five relations are verified:
    \begin{align*}
      \gamma^j \delta^{j,1} &= \id_{\sqparen{n-1}} & \text{($a_2d_1 = \id_C$)}\\
      \gamma^{j-1} \delta^{j-1,1} &= \id_{\sqparen{n-2}}  & \text{($p_2d_0 = \id_P$)}\\
      \gamma^i \delta^{i,1} &= \id_{\sqparen{n-1}} & \text{($a_1d'_2 = \id_B$)} \\
      \delta^{j-1,1}\gamma^i &= \gamma^i \delta^{j,1} & \text{($d_0p_1 =
        a_1d_1$)} \\
      \gamma^j \delta^{i,1} \delta^{j-1,1} \gamma^{j-1} &= 
      \gamma^j \delta^{j,1} \delta^{i,1} \gamma^{j-1} = 
      \delta^{i,1} \gamma^{j-1} = \gamma^j \delta^{i,1}
      & \text{($a_2 d'_2 d_0 p_2
      = a_2 d'_2$)} & \qedhere
    \end{align*}
  \end{enumerate}
\end{proof}

\begin{lemma}
  Suppose $n \ge 2$.  The diagram 
  \begin{equation*}
    \xymatrix{\sqparen{n-1} & \sqparen{n} \ar[l]_(0.35){\gamma^i}
    \ar[r]^(0.45){\sigma^j} & \sqparen{n-1}}
  \end{equation*}
  has an absolute pushout $\sqparen{\ell}$ in $\qsigma$ with maps $\sqparen{n-1}
  \to \sqparen{\ell}$ in $\qsigma^{-}$ and $n-2 \le \ell \le n-1$.
\end{lemma}
\begin{proof}
  We have four possibilities:
  \begin{enumerate}
  \item ($i > j$) The square
    \begin{equation*}
      \xymatrix{\sqparen{n} \ar[r]^{\sigma^j} \ar[d]_{\gamma^i} &
        \sqparen{n-1} \ar[d]^{\gamma^{i-1}} \\
        \sqparen{n-1} \ar[r]^{\sigma^j} & \sqparen{n-2}}
    \end{equation*}
    is an split pushout with sections
    \begin{align*}
      d_0 &= \delta^{i-1,1} &
      d_1 &= \delta^{i,1} &
      d'_2 &= \delta^{j,1}.
    \end{align*}
    The relations are satisfied:
    \begin{align*}
      \gamma^i \delta^{i,1} &= \id_{\sqparen{n-1}} & \text{($a_2d_1 = \id_C$)}\\
      \gamma^{i-1} \delta^{i-1} &= \id_{\sqparen{n-2}} & \text{($p_2d_0 = \id_P$)}\\
      \sigma^j \delta^{j,1} &= \id_{\sqparen{n-1}} & \text{($a_1d'_2 = \id_B$)} \\
      \delta^{i-1,1} \sigma^j &= \sigma^j \delta^{i,1} & \text{($d_0p_1 =  a_1d_1$)} \\
      \gamma^i \delta^{j,1} \delta^{i-1,1} \gamma^{i-1} &=
      \gamma^i \delta^{i,1} \delta^{j,1} \gamma^{i-1} =
      \delta^{j,1} \gamma^{i-1} = \gamma^i \delta^{j,1}
      & \text{($a_2 d'_2 d_0 p_2
      = a_2 d'_2$).}
    \end{align*}
  \item ($i = j$) The square
    \begin{equation*}
      \xymatrix{\sqparen{n} \ar[r]^{\gamma^i} \ar[d]_{\sigma^i} &
        \sqparen{n-1} \ar[d]^{\sigma^i} \\
        \sqparen{n-1} \ar[r]^{\sigma^i} & \sqparen{n-2}}
    \end{equation*}
    is an split pushout with sections
    \begin{align*}
      d_0 &= \delta^{i,0} &
      d_1 &= \delta^{i,0} &
      d'_2 &= \delta^{i+1,1}.
    \end{align*}
    The relations are satisfied:
    \begin{align*}
      \sigma^i \delta^{i,0} &= \id_{\sqparen{n-1}} & \text{($a_2d_1 = \id_C$)}\\
      \sigma^i \delta^{i,0} &= \id_{\sqparen{n-2}} & \text{($p_2d_0 = \id_P$)}\\
      \gamma^i \delta^{i+1,1} &= \id_{\sqparen{n-1}} & \text{($a_1d'_2 = \id_B$)} \\
      \delta^{i,0} \sigma^i &= \gamma^i \delta^{i,0} & \text{($d_0p_1 =  a_1d_1$)} \\
      \sigma^i \delta^{i+1,1} \delta^{i,0} \sigma^i &=
      \sigma^i \delta^{i,0} \delta^{i,1} \sigma^i =
      \delta^{i,1}\sigma^i = \sigma^i \delta^{i+1,1}
      & \text{($a_2 d'_2 d_0 p_2
      = a_2 d'_2$).}
    \end{align*}
  \item ($i+1 = j$) The square
    \begin{equation*}
      \xymatrix{\sqparen{n} \ar[r]^{\gamma^i} \ar[d]_{\sigma^j} &
        \sqparen{n-1} \ar[d]^{\sigma^i} \\
        \sqparen{n-1} \ar[r]^{\sigma^i} & \sqparen{n-2}}
    \end{equation*}
    is an split pushout with sections
    \begin{align*}
      d_0 &= \delta^{i,0} &
      d_1 &= \delta^{j,0} &
      d'_2 &= \delta^{i,1}.
    \end{align*}
    The relations are satisfied:
    \begin{align*}
      \sigma^j \delta^{j,0} &= \id_{\sqparen{n-1}} & \text{($a_2d_1 = \id_C$)}\\
      \sigma^i \delta^{i,0} &= \id_{\sqparen{n-2}} & \text{($p_2d_0 = \id_P$)}\\
      \gamma^i \delta^{i,1} &= \id_{\sqparen{n-1}} & \text{($a_1d'_2 = \id_B$)} \\
      \delta^{i,0} \sigma^i &= \gamma^i \delta^{j,0} & \text{($d_0p_1 =  a_1d_1$)} \\
      \sigma^j \delta^{i,1} \delta^{i,0} \sigma^i &=
      \sigma^j \delta^{j,0} \delta^{i,1} \sigma^i =
      \delta^{i,1}\sigma^i = \sigma^j \delta^{i,1} & \text{($a_2 d'_2 d_0 p_2
      = a_2 d'_2$)}
    \end{align*}
  \item ($i +1 < j$) The square
    \begin{equation*}
      \xymatrix{\sqparen{n} \ar[r]^{\gamma^i} \ar[d]_{\sigma^j} &
        \sqparen{n-1} \ar[d]^{\sigma^{j-1}} \\
        \sqparen{n-1} \ar[r]^{\gamma^i} & \sqparen{n-2}}
    \end{equation*}
    is an split pushout with sections
    \begin{align*}
      d_0 &= \delta^{j-1,0} &
      d_1 &= \delta^{j,0} &
      d'_2 &= \delta^{i,1}.
    \end{align*}
    The relations are satisfied:
    \begin{align*}
      \sigma^j \delta^{j,0} &= \id_{\sqparen{n-1}} & \text{($a_2d_1 =
        \id_C$)} \\
      \sigma^{j-1} \delta^{j-1,0} &= \id_{\sqparen{n-2}} &
      \text{($p_2d_0 = \id_P$)} \\
      \gamma^i \delta^{i,1} &= \id_{\sqparen{n-1}} & \text{($a_1d'_2 = \id_B$)} \\
      \delta^{j-1,0} \gamma^i &= \gamma^i \delta^{j,0} & \text{($d_0p_1 =  a_1d_1$)} \\
      \sigma^j \delta^{i,1} \delta^{j-1,0} \sigma^{j-1} &=
      \sigma^j \delta^{j,0} \delta^{i,1} \sigma^{j-1} =
      \delta^{i,1}\sigma^{j-1} = \sigma^j \delta^{i,1}  & \text{($a_2 d'_2 d_0 p_2
      = a_2 d'_2$).}&\qedhere
    \end{align*}
  \end{enumerate}
\end{proof}

\begin{corollary}
  The category $\qsigma$ satisfies axiom EZ3 of Definition
  \ref{def:ezcategory}.
\end{corollary}

\section{The symmetric cubical site models the homotopy category}
In this section, we'll equip $\qsSet$ with a model
structure Quillen equivalent to $\sSet$.  This is the heart of
the paper.  We'll start by describing a spatial model structure on
$\qSet$.  We will then lift the model structure from
$\qSet$ along the restriction functor
$i^\ast:\qsSet \to \qSet$.  In order to do this, we
need to check that cell complexes in $\qsSet$ built out of
the representable functors are well-behaved homotopically.  The
outline of the argument is standard; as usual, it requires some work
to verify.  The resulting Quillen pair $i_! \adjoint i^\ast$ is then
readily shown to be a Quillen equivalence.  Finally, we'll discuss the
monoidal properties of the lifted model structure on
$\qsSet$.

\subsection{The homotopy theory of cubical sets}
In \cite{cisinskithesis}, Cisinski proves that $\QQ$ is a test
category and thus the category $\qSet = \Cat{Set}^{\QQ^\Op}$
models spaces.  Jardine gives a summary of cubical homotopy theory
from Cisinski's perspective in \cite{jardinecubes}.  We'll summarize
their results here.  Recall that $\partial \square^n$ is the
subpresheaf of $\square^n$ given by
\begin{equation*}
  (\partial \square^n)_m = \big\{ f:\sqparen{m} \to \sqparen{n} \in
  \QQ \bigm| \text{$f$ factors as $f:\sqparen{m} \to \sqparen{k} \to \sqparen{n}$, $k < n$}\big\}.
\end{equation*}
This comes equipped with a monomorphism $\partial\square^n \to
\square^n$.  Put another way, $\partial \square^n$ is the union of
the $(n-1)$-dimensional faces of $\square^n$.  We define the
\emph{$i,\varepsilon$-cap}
\begin{equation*}
  (\sqcap^n_{i,\varepsilon})_m = \big\{ f:\sqparen{m} \to \sqparen{n}
  \in \QQ \bigm|
  \text{$f$ factors as $f:\sqparen{m} \to \sqparen{n-1}
    \xrightarrow{d} \sqparen{n}$, $d\ne \delta^{\varepsilon,i}_n$} \big\}
\end{equation*}
for $1\le i\le n$.
This comes equipped with a monomorphism $\sqcap^n_{i,\varepsilon} \to
\partial \square^n$.
\begin{definition}
  We say a functor $F:\mathscr{B} \to \mathscr{C}$ of small categories
  is a \emph{Thomason equivalence} if $F$ induces a weak equivalence
  $\Nerve F : \Nerve \mathscr{B} \to \Nerve \mathscr{C}$ on nerves.
  Let $\mathscr{A}$ be a small category.  We say a map $f:X\to Y$ in
  $\widehat{\mathscr{A}}$ is an \emph{$\infty$-equivalence} if $f$
  induces a Thomason equivalence
  \begin{equation*}
    \mathscr{A}\downarrow f : \mathscr{A} \downarrow X \to \mathscr{A}\downarrow Y
  \end{equation*}
  of categories.
\end{definition}
\begin{definition}
  The \emph{simplicial realization} of a cubical set $X\in\qSet$ is
  the colimit
  \begin{equation*}
    |X| = \colim_{\square^n \to X} (\Delta[1])^n
  \end{equation*}
  of simplicial sets.
\end{definition}
Note that simplicial realization is the unique cocontinuous functor
$\qSet \to \sSet$ taking $\square^n$ to $(\Delta[1])^n$.  Since its
restriction to $\QQ$ is strong monoidal, it is strong monoidal on
$\qSet$.  We can now state the following theorem:
\begin{theorem}[{\cite[Th\'eor\`eme
    8.4.38]{cisinskithesis}}]\label{theorem:cisinskimodelstructureonqq}
  \begin{enumerate}
  \item The category $\qSet$ forms a proper model category
    with cofibrations monomorphisms and weak equivalences the
    $\infty$-equivalences.  We call this model structure the
    \emph{spatial model structure}.  It is cofibrantly generated with
    generating cofibrations
    \begin{equation*}
      \{ \partial \square^n \to \square^n \mid n \ge 0 \}
    \end{equation*}
    and generating acyclic cofibrations
    \begin{equation*}
      \{ \sqcap^n_{i,\varepsilon} \to \square^n \mid 1\le i\le n, \varepsilon=0,1 \}.
    \end{equation*}
  \item The spatial model structure is monoidal: if $i:A\to B$ and
    $j:K\to L$ are cofibrations,
    \begin{equation*}
      i\odot j:A\otimes L \amalg_{A\otimes K} B\otimes K \to
      B\otimes L
    \end{equation*}
    is a cofibration, acyclic if either $i$ or $j$ is.
  \item Simplicial realization is a left Quillen equivalence
    $\qSet\to \sSet$.
  \end{enumerate}
\end{theorem}
Theorem \ref{theorem:cisinskimodelstructureonqq} is the basis of
everything that follows.  We will take it for granted.  Jardine also
gives a proof of it in the survey \cite{jardinecubes} following
Cisinski's methods.

\subsection{Homotopy and asphericity}
Recall that the categories $\QQ$ and $\qsigma$ are related by an
inclusion functor $i:\QQ\to \qsigma$.  This produces an adjoint pair
\begin{equation*}
  \xymatrix{i_! : \qSet \ar@<0.5ex>[r] &
    \qsSet:i^\ast \ar@<0.5ex>[l]}
\end{equation*}
given by left Kan extension and restriction.
\begin{proposition}\label{prop:ibangismonoidal}
  The functors $i_!$ and $i^\ast$ are strong and lax monoidal,
  respectively.
\end{proposition}
\begin{proof}
  That $i_!$ is strong monoidal is a consequence of Proposition
  \ref{prop:imkelly}: since the square
  \begin{equation*}
    \xymatrix{\ar@{}[dr] |\Downarrow \QQ \ar[r]^i \ar[d]_{[{-}]} & \qsigma \ar[d]^{[-]} \\
      \qSet \ar[r]_{i_!} & \qsSet}
  \end{equation*}
  commutes up to natural isomorphism and $i$ is strong monoidal, the
  extension $i_!$ is strong monoidal.  Now suppose $K$ and $L$ are
  extended cubical sets.  The counit of the adjunction $i_! \adjoint
  i^\ast$ together with the monoidalness of $i_!$ yields a natural map
  \begin{equation*}
    \xymatrix{i_! (i^\ast K \otimes i^\ast L) \ar[r]^\sim & i_!i^\ast K \otimes i_!
      i^\ast L \ar[r] & K\otimes L}
  \end{equation*}
  The adjoint is a natural transformation
  \begin{equation*}
    i^\ast K \otimes i^\ast L \to i^\ast (K\otimes L)
  \end{equation*}
  making $i^\ast$ lax monoidal.
\end{proof}

\begin{definition}
  Suppose $n > 0$.  Let $\{\varepsilon\}$ denote the formal
  $(0,n)$-product
  \begin{equation*}
    (\underbrace{\varepsilon, \dotsc, \varepsilon}_\text{$n$ entries})
  \end{equation*}
  for $\varepsilon=0,1$.  Note that $\{\varepsilon\} =
  (d^{1,\varepsilon})^n$.  Suppose $f,g:X\to Y$ are two maps in
  $\qsSet$.  We say $f$ and $g$ are
  \emph{$\square_\Sigma^n$-homotopic} if there is a filler $h$ in the
  diagram
  \begin{equation*}
    \xymatrix{X \otimes \square^0_\Sigma \ar[d]_{\id\otimes\{0\}}
    \ar[rd]^f \\
    X \otimes \square^n_\Sigma \ar@{.>}[r]^h & Y. \\
    X \otimes \square^0_\Sigma \ar[u]^{\id\otimes\{1\}} \ar[ur]^g}
  \end{equation*}
  We call $h$ a \emph{$\square^n_\Sigma$-homotopy} from $f$ to $g$.
  By abuse of terminology, we'll sometimes simply call $f$ and $g$
  \emph{homotopic.}  We say a map $k$ is a \emph{homotopy equivalence}
  if there is a map $\ell$ so that $k\ell$ and $\ell k$ are both
  homotopic to the identity.
  We define $\square^n$-homotopy in $\qSet$ similarly.  
\end{definition}
Note that $\square_\Sigma^n$-homotopy is not an equivalence
relation for arbitrary $Y$---the (extended) cubical set $Y$ must
possess a sort of homotopy extension property, i.e., $Y$ must be
fibrant.  This is precisely the same reason that $\Delta[1]$-homotopy
is not an equivalence relation on maps in $\sSet$ unless the maps
have a Kan complex as their target.  Using the spatial model structure
on $\qSet$, we have the following standard result:

\begin{proposition}
  Suppose $k:X\to Y$ is a homotopy equivalence in $\qSet$.
  Then $k$ is an $\infty$-equivalence.
\end{proposition}

\begin{lemma}\label{lemma:iastpreserveshomotopy}
  Suppose $f$ and $g:X\to Y$ are $\square^n_\Sigma$-homotopic maps in
  $\qsSet$.  Then $i^\ast f$ and $i^\ast g$ are
  $\square^n$-homotopic.
\end{lemma}
\begin{proof}
  Let $h:X\otimes \square^n_\Sigma \to Y$ be a homotopy from $f$ to
  $g$.  By Proposition \ref{prop:ibangismonoidal}, $i^\ast$ is lax
  monoidal, so we have a diagram
  \begin{equation*}
    \xymatrix@C=4pc{i^\ast X \otimes \square^0 \ar[r]^\sim \ar[d]_{\id\otimes\{0\}} &
      i^\ast X \otimes i^\ast \square^0_\Sigma \ar[r]^\sim
      \ar[d]_{\id\otimes i^\ast\{0\}} &
      i^\ast(X \otimes \square^0_\Sigma) \ar[d]_{i^\ast (\id\otimes
      \{0\})} \ar[rd]^{i^\ast f} \\
      i^\ast X \otimes \square^n \ar[r] &
      i^\ast X \otimes i^\ast \square^n_\Sigma \ar[r] &
      i^\ast(X \otimes \square^n_\Sigma) \ar[r]^{i^\ast h} &
      i^\ast Y. \\
      i^\ast X \otimes \square^0 \ar[r]^\sim \ar[u]^{\id\otimes\{1\}} &
      i^\ast X \otimes i^\ast \square^0_\Sigma \ar[r]^\sim
      \ar[u]^{\id\otimes i^\ast\{1\}} &
      i^\ast(X \otimes \square^0_\Sigma)
      \ar[u]^{i^\ast(\id\otimes\{1\})} \ar[ru]_{i^\ast g}}
  \end{equation*}
  The unit $\square^0 \to i^\ast
  \square^0_\Sigma$ is an isomorphism since $\square^0_\Sigma$ and
  $\square^0$ are terminal and $i^\ast$ is a right adjoint.  Hence the
  top and bottom horizontal arrows are isomorphisms.  As a result, the
  horizontal chain of arrows is a $\square^n$-homotopy between $i^\ast
  f$ and $i^\ast g$.
\end{proof}

\begin{corollary}\label{cor:iastpreserveshoequivs}
  The functor $i^\ast$ preserves homotopy equivalences.
\end{corollary}

\begin{lemma}\label{lemma:iisaspherical}
  The inclusion functor $i:\QQ \to \qsigma$ is aspherical, i.e., for
  all $n \ge 0$,
  \begin{equation*}
    i \downarrow \sqparen{n} \to \qsigma \downarrow \sqparen{n}
  \end{equation*}
  is a Thomason equivalence.
\end{lemma}
\begin{proof}
  We define a map $H:\sqparen{2n} \to \sqparen{n}$ as the formal product
  \begin{equation*}
    (x_1 \wedge x_{n+1},x_2 \wedge x_{n+2},\dotsc,x_n\wedge x_{2n}).
  \end{equation*}
  This is an application of a symmetry followed by $n$ conjunctions.
  The map $H$ gives a homotopy between $\{0\}$ and the identity map on
  $\square^n_\Sigma$:
  \begin{equation*}
    \xymatrix@C=4pc{\square_\Sigma^n \ar[d]_{\id\otimes\{0\}} \ar[dr]^{\{0\}} \\
      \square_\Sigma^n\otimes \square_\Sigma^n \ar[r]^H & \square_\Sigma^n \\
      \square_\Sigma^n \ar[u]^{\id\otimes\{1\}} \ar[ur]_{\id}}
  \end{equation*}
  Hence the inclusion $\{0\}$ is a homotopy equivalence
  in $\qsSet$.  By Corollary \ref{cor:iastpreserveshoequivs},
  \begin{equation*}
    i^\ast\{0\}:\square^0 \to i^\ast \square_\Sigma^n
  \end{equation*}
  is a homotopy equivalence and hence $\infty$-equivalence in
  $\qSet$.  Thus
  \begin{equation*}
    \QQ \to \QQ\downarrow i^\ast \square_\Sigma^n
  \end{equation*}
  is a Thomason equivalence and $\Nerve(i^\ast \square_\Sigma^n)$ is
  contractible.  Since $\QQ \downarrow i^\ast \square_\Sigma^n$ is
  equivalent to $i \downarrow \square_\Sigma^n$, we conclude that
  \begin{equation*}
    i \downarrow \sqparen{n} \to \qsigma \downarrow \sqparen{n}
  \end{equation*}
  is a Thomason equivalence.
\end{proof}

\begin{proposition}\label{prop:iastreflectsweakequivalences}
  Suppose $X\in\qsSet$.  Then the functor
  \begin{equation*}
    \QQ \downarrow i^\ast X \to \qsigma \downarrow X
  \end{equation*}
  induced by $i$ induces an equivalence of nerves.
\end{proposition}
\begin{proof}
  This is a special case of \cite[Proposition 1.2.9]{maltsiniotis}.
  Suppose $s:\square^n_\Sigma \to X$ is a cube of $X$.  Consider
  the category
  \begin{equation*}
    (\QQ\downarrow i^\ast X)\downarrow s:
  \end{equation*}
  this is the category of triangles
  \begin{equation*}
    \xymatrix{\square_\Sigma^m \ar[r] \ar[d] & X \\
      \square_\Sigma^n \ar[ur]_s}
  \end{equation*}
  with morphisms diagrams of the shape
  \begin{equation*}
    \xymatrix{\square_\Sigma^{m'} \ar[r]^{i(f)} \ar[dr] & \square_\Sigma^m \ar[r] \ar[d] & X \\
      &\square_\Sigma^n \ar[ur]_s}
  \end{equation*}
  The functor
  \begin{equation*}
    (\QQ\downarrow i^\ast X)\downarrow s \to \QQ \downarrow \square^n_\Sigma
  \end{equation*}
  forgetting the map to $X$ has a left adjoint given by composition
  with $s$, so it is a Thomason equivalence.  By Lemma
  \ref{lemma:iisaspherical}, we may conclude that
  \begin{equation*}
    (\QQ \downarrow i^\ast X) \downarrow s
  \end{equation*}
  has contractible nerve, so by Quillen's Theorem A
  \cite{quillenktheory},
  \begin{equation*}
    \QQ \downarrow i^\ast X \to \qsigma \downarrow X
  \end{equation*}
  is a Thomason equivalence.
\end{proof}
\begin{corollary}\label{cor:iastreflectseqs}
  The functor $i^\ast$ reflects $\infty$-equivalences; i.e., $X\to Y$
  in $\qsSet$ induces a Thomason equivalence
  \begin{equation*}
    \qsigma \downarrow X \to \qsigma \downarrow Y
  \end{equation*}
  if and only if
  \begin{equation*}
    \QQ \downarrow i^\ast X \to \QQ \downarrow i^\ast Y
  \end{equation*}
  is a Thomason equivalence.
\end{corollary}

\subsection{A Quillen equivalence}
We will show that $i_! \adjoint i^\ast$ is a Quillen equivalence
simultaneously with the construction of the spatial model structure on
$\qsSet$.
\begin{proposition}\label{prop:ibangisleftquillen}
  The functor $i_!:\qSet \to \qsSet$ preserves
  monomorphisms.
\end{proposition}
Before we embark
on this, recall that $\partial \square^n_\Sigma$ is the subpresheaf of
$\square^n_\Sigma$ given by
\begin{equation*}
  \partial\square^n_\Sigma(\sqparen{m}) = \big\{ f\in \qsigma(\sqparen{m}, \sqparen{n}) \bigm|
  \text{$f$ factors $\sqparen{m} \to \sqparen{k} \to \sqparen{n}$, $k
  < n$ in $\qsigma$} \big\}.
\end{equation*}
Another description of $\partial \square^n_\Sigma(\sqparen{m})$ is as
the set of formal $(m,n)$-products with at least one entry $0$ or $1$.
As we'll describe below, $\partial \square^n_\Sigma$ is the union of
the faces of $\square^n_\Sigma$.
\begin{proof}[Proof of Proposition \ref{prop:ibangisleftquillen}]
  By Corollary \ref{cor:skeletoncomparison}, the map $i_! \sk_{n-1}
  \square^n \to \sk_{n-1} \square^n_\Sigma$ is an isomorphism.  This
  implies that $i_!(\partial \square^n \to \square^n)$ is, up to
  isomorphism, the map $\partial \square^n_\Sigma \to
  \square^n_\Sigma$.  Since
  \begin{equation*}
    \mathsf{mono} = \Cell \big\{ \partial \square^n \to \square^n
    \bigm| n \ge 0 \big\},
  \end{equation*}
  we may conclude that $i_!$ preserves all monomorphisms in
  $\qSet$.
\end{proof}

\begin{lemma}\label{lemma:cofibranttower}
  Suppose $Y$ is an $n$-skeletal cubical set.  The map
  \begin{equation*}
    Y \amalg_{\sk_{n-1} Y} i^\ast i_! \sk_{n-1} Y \to i^\ast i_! Y
  \end{equation*}
  is a monomorphism.
\end{lemma}
\begin{proof}
  First note that the corner map
  \begin{equation*}
    i^\ast \partial \square^n_\Sigma \amalg_{\partial \square^n}
    \square^n \to i^\ast \square^n_\Sigma
  \end{equation*}
  is a monomorphism.
  This is a consequence of the fact that $i$ is faithful: for $\sqparen{m}$,
  we have
  \begin{equation*}
    \big(i^\ast \partial \square^n_\Sigma \amalg_{\partial \square^n}
    \square^n\big)_m = 
    \big\{ f:\sqparen{m} \to \sqparen{n} \bigm|
    \text{$f$ factors through $\sqparen{k}$, $k < n$, or $f\in\Ar\QQ$}\big\}.
  \end{equation*}
  Let $S$ be the set of nondegenerate $n$-simplices of $Y$.  By
  Proposition \ref{prop:cellularmodelforez} and Proposition
  \ref{prop:qqisez}, we may write $Y$ as a pushout
  \begin{equation*}
    \xymatrix{ \coprod_S \partial \square^n \ar[r] \ar[d] & \sk_{n-1} Y
    \ar[d] \\
    \coprod_S \square^n \ar[r] & Y}
  \end{equation*}
  where $S$ is the set of nondegenerate $n$-simplices of $Y$.  Write
  $\eta:\id\to i^\ast i_!$ for the unit of the adjunction $i_!\adjoint
  i^\ast$.  Consider the cube
  \begin{equation}\label{eq:commutativecube}
    \vcenter{
      \xymatrix{ & \coprod_S i^\ast i_! \partial \square^n \ar[rr]
        \ar'[d][dd] && i^\ast i_! \sk_{n-1} Y
        \ar[dd] \\
        \coprod_S \partial \square^n \ar[rr] \ar[ur]^{\coprod_S
          \eta_{\partial\square^n}} \ar[dd]_{j} && \sk_{n-1} Y
        \ar[dd] \ar[ur]^{\eta_{\sk_{n-1} Y}}
        \\
        & \coprod_S i^\ast i_! \square^n \ar'[r][rr] && i^\ast i_! Y. \\
        \coprod_S \square^n \ar[ur]^{\coprod_S \eta_{\square^n}}
        \ar[rr] && Y \ar[ur]^{\eta_Y}
      }}
  \end{equation}
  The functors $i^\ast$ and $i_!$ both preserve colimits, so the front
  and back faces are both pushouts.  As a result, the square
  \begin{equation*}
    \xymatrix{ \coprod_S \big( i^\ast \partial \square^n_\Sigma \amalg_{\partial \square^n}
      \square^n \big) \ar[r] \ar[d] & i^\ast i_! \sk_{n-1} Y \amalg_{\sk_{n-1} Y}
      Y \ar[d]_g \\
      \coprod_S i^\ast i_! \square^n \ar[r] & i^\ast i_! Y}
  \end{equation*}
  is a pushout; since $\qSet$ is a topos, the arrow $g$ is a
  monomorphism.
\end{proof}

\begin{lemma}\label{lemma:weakeqclosedundercoprod}
  An arbitrary small coproduct of $\infty$-equivalences in $\qSet$
  is an $\infty$-equivalence.
\end{lemma}
\begin{proof}
  This is a standard model category result (Ken Brown's lemma
  \cite{hovey}, together with the fact that everything in
  $\qSet$ is cofibrant).  Alternatively, observe that $|{-}|$
  reflects weak equivalences and preserves small coproducts; arbitrary
  small coproducts  of weak equivalences in $\sSet$ are themselves
  weak equivalences.
\end{proof}

\begin{proposition}\label{prop:ibangiastunitiswkeq}
  The unit $\eta:\id \to i^\ast i_!$ of the adjunction $i_! \adjoint
  i^\ast$ is a natural $\infty$-equivalence in $\qSet$.
\end{proposition}
\begin{proof}
  We'll first prove that $\eta_X$ is an $\infty$-equivalence for
  skeletal $X$ by induction on the dimension.  If $X$ is $0$-skeletal,
  then $X = \coprod_S \square^0$ and $\eta_X$ is an isomorphism.  Let
  $n > 0$ and suppose $\eta_X$ is a weak equivalence for all
  $(n-1)$-skeletal $X$.  In particular, $\eta_{\partial \square^n}$ is
  an $\infty$-equivalence since $\partial \square^n$ is the
  $(n-1)$-skeleton of $\square^n$.  Suppose $Y$ is $n$-skeletal.  From
  Corollary \ref{cor:iastreflectseqs}, we know that $\square^n \to
  i^\ast \square^n_\Sigma$ is an $\infty$-equivalence.  Recall the
  cube \eqref{eq:commutativecube} in the proof of Lemma
  \ref{lemma:cofibranttower}.  The front and back faces are both
  pushout squares.  The arrows $j$ and $i^\ast i_! j$ are both
  monomorphisms by Proposition \ref{prop:ibangisleftquillen}.  Every
  object of $\qSet$ is cofibrant, so these pushout squares are
  both homotopy cocartesian.  By Lemma
  \ref{lemma:weakeqclosedundercoprod} and our assumptions, the
  diagonal arrows $\coprod_S \eta_{\square^n}$, $\coprod_S
  \eta_{\partial\square^n}$ and $\eta_{\sk_{n-1} Y}$ are weak
  equivalences, so $\eta_Y$ is an $\infty$-equivalence.  By induction,
  we may conclude that $\eta_X$ is an $\infty$-equivalence for all $X$
  which are $n$-skeletal for some $n$.

  Suppose $X$ is an arbitary cubical set.  Now consider the ladder
  \begin{equation*}
    \xymatrix{\sk_0 X \ar[d]^{\eta_0} \ar[r] & \sk_1 X \ar[r] \ar[d]^{\eta_1}
    & \dotsb \ar[r] & \sk_n X \ar[r] \ar[d]^{\eta_n} & \dotsb \\
    i^\ast i_! \sk_0 X \ar[r] & i^\ast i_! \sk_1 X \ar[r] & \dotsb
    \ar[r]  & i^\ast i_! \sk_n X \ar[r] & \dotsb}
  \end{equation*}
  By Lemma \ref{lemma:cofibranttower}, this map is an acyclic Reedy
  cofibration: the map $\eta_0$ is a cofibration, the corner maps
  \begin{equation*}
    i^\ast i_!\sk_{n-1} X \amalg_{\sk_{n-1} X} \sk_n X \to i^\ast i_!
    \sk_n X
  \end{equation*}
  are cofibrations, and each $\eta_i$ is an $\infty$-equivalence.  Hence
  the colimit $\eta_X : X \to i^\ast i_!X$ is an
  $\infty$-equivalence.
\end{proof}
We could have avoided using Lemma \ref{lemma:cofibranttower} by
Reedy's Theorem C \cite{reedy}.

\begin{corollary}\label{cor:icounitpreservesweakeqs}
  The counit $\varepsilon:i_!i^\ast \to \id$ of the adjunction
  $i_!\adjoint i^\ast$ is a natural $\infty$-equivalence in $\qsigma$.
\end{corollary}
\begin{proof}
  Suppose $X\in\qsSet$.  Consider the triangle
  \begin{equation*}
    \xymatrix{i^\ast X \ar[rd]_\id \ar[r]^{\eta_{i^\ast X}} & i^\ast
    i_! i^\ast X \ar[d]^{i^\ast \varepsilon_X} \\
    & i^\ast X.}
  \end{equation*}
  The map $\eta_{i^\ast X}$ is an $\infty$-equivalence by Proposition
  \ref{prop:ibangiastunitiswkeq}, so $i^\ast \varepsilon_X$ is an
  $\infty$-equivalence.  By Proposition
  \ref{prop:iastreflectsweakequivalences}, $\varepsilon_X$ is an
  $\infty$-equivalence.
\end{proof}

We can finally prove the main theorem of this paper.

\begin{theorem}\label{theorem:extendedcubicalmodelstructure}
  The category $\qsSet$ forms a left proper cofibrantly
  generated model category known as the \emph{spatial model structure}
  with weak equivalences the $\infty$-equivalences and fibrations the
  maps $p:X\to Y$ which are fibrations in $\qSet$ upon
  application of the restriction $i^\ast$.  The set of generating
  cofibrations is
  \begin{equation*}
    I = \{ \partial \square_\Sigma^n \to \square_\Sigma^n \mid n \ge 0\}
  \end{equation*}
  and the set of generating acyclic cofibrations is
  \begin{equation*}
    J = \{ i_! \sqcap_{j,\varepsilon}^n \to \square_\Sigma^n \mid
    \text{$1\le j\le n$ and $\varepsilon=0,1$}\}.
  \end{equation*}
\end{theorem}
\begin{proof}
  This is a consequence of a standard result on lifting model
  structures along an adjunction---see, for example,
  \cite{schwedeshipley}.  The key point is the following: suppose
  \begin{equation*}
    \xymatrix{i_! \sqcap_{j,\varepsilon}^n \ar[r] \ar[d]_{i_! e} & A \ar[d]^f
    \\
    i_! \square^n \ar[r] & B}
  \end{equation*}
  is a pushout in $\qsSet$.  The functor $i^\ast$ preserves
  all colimits and limits, so the right square in
  \begin{equation*}
    \xymatrix{\sqcap_{j,\varepsilon}^n \ar[r] \ar[d]_{e} & 
      i^\ast i_! \sqcap_{j,\varepsilon}^n \ar[r] \ar[d]_{i^\ast i_! e} &
      i^\ast A \ar[d]^{i^\ast f}
      \\
      \square^n \ar[r] & i^\ast i_! \square^n \ar[r] & i^\ast B}
  \end{equation*}
  is a pushout in $\qSet$.  But by Proposition
  \ref{prop:ibangiastunitiswkeq}, $i^\ast i_! e$ is a weak
  equivalence.  By Proposition \ref{prop:ibangisleftquillen}, $i^\ast
  i_! e$ is a monomorphism.  Hence $i^\ast f$ is an
  $\infty$-equivalence in $\qSet$, so by Corollary
  \ref{cor:iastreflectseqs} $f$ is an $\infty$-equivalence.
  Since
  $\qsSet$ is locally presentable, we can use the
  small object argument to factor every arrow in $\qsSet$
  as a map in $\Cell J$ followed by a $J$-injective map \cite{beke1}.
  But by the above discussion---together with the fact that $i^\ast$
  preserves filtered colimits---the maps in $\Cell J$ are acyclic
  cofibrations and the maps in $\inj J$ are fibrations.
  For left properness, apply the functor $i^\ast$ to the necessary
  diagram and note that $i^\ast$ preserves cofibrations.
\end{proof}

Since $\Delta[1]$ is a cubical monoid, we may define the \emph{extended
  geometric realization} functor $|{-}|_\Sigma$ to be the unique
cocontinuous strong monoidal functor $\qsigma \to \sSet$ taking
$\square_\Sigma^1$ to $\Delta[1]$.

\begin{theorem}\label{theorem:quillenequivalences}
  The functors $i_!$ and $|{-}|_\Sigma$ are both left Quillen
  equivalences.  The diagram
  \begin{equation}\label{eq:quillenequivalencesdiagram}
    \vcenter{
      \xymatrix{\qSet \ar[rr]^{i_!} \ar[rd]_{|{-}|} &&
        \qsSet \ar[ld]^{|{-}|_\Sigma} \\
        & \sSet
      }}
  \end{equation}
  commutes up to natural isomorphism.
\end{theorem}
\begin{proof}
  We've proved that the unit and counit $\eta:\id \to i^\ast i_!$ and
  $\varepsilon: i_! i^\ast \to \id$ are natural $\infty$-equivalences
  (\ref{prop:ibangiastunitiswkeq} and Corollary
  \ref{cor:icounitpreservesweakeqs}).  Strictly speaking, this is not
  the right condition for Quillen equivalences, as we'd need to use
  the derived left and right adjoints.  However, $i^\ast$ and $i_!$
  coincide with their derived functors: $i_!$ is left Quillen, but
  everything in $\qSet$ is cofibrant, so it preserves all
  $\infty$-equivalences.  The rights adjoint $i^\ast$ preserves all
  $\infty$-equivalences as well (Proposition
  \ref{prop:iastreflectsweakequivalences}).  Two-out-of-three ensures
  that $|{-}|_\Sigma$ is a left Quillen equivalence.
\end{proof}

\begin{remark}
  A few notes about Theorem
  \ref{theorem:extendedcubicalmodelstructure} are in order.
  Not all monomorphisms in the spatial model structure on
  $\qsSet$ are cofibrations.  For example, the
  $\Sigma_n$-orbits of $\partial \square_\Sigma^n \to
  \square_\Sigma^n$ is a monomorphism, but if $n > 1$, it is not a
  cofibration.  As a result $|{-}|_\Sigma$ may not reflect weak
  equivalences---however, its left derived functor
  $\mathbf{L}|{-}|_\Sigma$ preserves and reflects weak equivalences.
  
  We've shown that $i^\ast$ is a left and right Quillen functor.  On the
  level of homotopy categories, since $i_!  \adjoint i^\ast$ induces
  an equivalence of $\Ho \qSet$ with $\Ho \qsSet$,
  the adjoint pair $i^\ast \adjoint \mathbf{R}i_\ast$ must also induce
  an equivalence of $\Ho \qSet$ with $\Ho \qsSet$,
  and so $\mathbf{R}i_\ast$ and $i_!$ coincide.
\end{remark}

\subsection{The extended product}
We've now shown that $\qsSet$ models spaces.  In the
remainder of this section, we'll prove that the monoidal structure on
$\qsSet$ is compatible with the spatial model structure.
\begin{lemma}\label{lemma:projectionequivalence}
  Suppose $X$ is an extended cubical set and $n \ge 0$.  The map
  $\pi:X\otimes \square^n_\Sigma \to X$ given by the product of the
  identity map on $X$ and the unique map $\square^n_\Sigma \to
  \square^0_\Sigma$ is an $\infty$-equivalence.
\end{lemma}
\begin{proof}
  This is essentially the same as the proof of Lemma
  \ref{lemma:iisaspherical}.  It is sufficient to prove for all $X$
  when $n = 1$.  Let $IX = X\otimes \square^1_\Sigma$ and let $s$ be the map
  \begin{equation*}
    s = \id_X \otimes \{0\} : X \to IX
  \end{equation*}
  Then $\pi s = \id_X$.  We have a homotopy
  \begin{equation*}
    \xymatrix@C=4pc{
      X \otimes \square^1_\Sigma \ar[d]_{\id_{IX}\otimes\{0\}}
      \ar[dr]^{s\pi} \\
      X \otimes \square^2_\Sigma \ar[r]^{\id_X \otimes
        \gamma^1} & X \otimes \square^1_\Sigma \\
      X \otimes \square^1_\Sigma \ar[u]^{\id_{IX}\otimes\{1\}}
      \ar[ur]_{\id_{IX}}}
  \end{equation*}
  between $\id_{IX}$ and $s\pi$, so $\pi$ is a
  homotopy equivalence.  By Corollary \ref{cor:iastpreserveshoequivs},
  $i^\ast\pi$ is a homotopy equivalence and hence
  $\infty$-equivalence, so $\pi$ is an $\infty$-equivalence.
\end{proof}

\begin{lemma}\label{lemma:extendedproductpreservesmonos}
  Suppose $i:A\to B$ and $j:K\to L$ are monomorphisms in
  $\qsSet$.  Then the pushout-product
  \begin{equation*}
    i \odot j : A\otimes L \amalg_{A\otimes K} B\otimes K \to
    B\otimes L
  \end{equation*}
  is a monomorphism.
\end{lemma}
\begin{proof}
  Recall from Proposition \ref{prop:cellularmodelforqsigma} that
  \begin{equation*}
    I_\Sigma = \big\{ (\partial \square_\Sigma^n)_H \to
    (\square^n_\Sigma)_H \bigm| \text{$n \ge 0$ and $H\leqslant
    \Sigma_n$} \big\}\notag
  \end{equation*}
  is a cellular model for $\qsSet$.  First, we'll show
  that the pushout-product of any two maps in $I_\Sigma$ is a
  monomorphism.  Suppose $n,m\ge 0$ and $H_n$, $H_m$ subgroups of
  $\Sigma_n$ and $\Sigma_m$, respectively.  Let $H = H_n\times H_m$ be
  the subgroup of $\Sigma_{n+m}$ generated by the images of $H_n$ and
  $H_m$ under the homomorphism $\Sigma_n \times \Sigma_m \to
  \Sigma_{n+m}$.  Then 
  \begin{equation*}
    \big( \partial \square_\Sigma^n \to \square_\Sigma^n)_{H_n}
    \odot 
    \big( \partial \square_\Sigma^m \to \square_\Sigma^m)_{H_m}
    \cong 
    \big( \partial \square_\Sigma^{n+m} \to \square_\Sigma^{n+m})_H.
  \end{equation*}
  This map is a monomorphism.  A standard deduction lets us upgrade
  this to deduce that the pushout-product of any two monomorphisms is
  a monomorphism: by the small object argument applied to $I_\Sigma$,
  we know that $j$ is a monomorphism if and only if $j \pitchfork p$
  for all $p\in\inj I_\Sigma$ \cite{beke1}.
\end{proof}

\begin{theorem}
  The spatial model structure on $\qsSet$ is monoidal and satisfies
  the Schwede-Shipley monoid axiom.
\end{theorem}
\begin{proof}
  That the spatial model structure is monoidal is a straightforward
  consequence of the fact that the generating cofibrations and acyclic
  cofibrations are given by left Kan extension along $i$, which is
  itself strong monoidal.  That is, if $f$ and $g$ are cofibrations
  in $\qSet$, then $i_! f \odot i_! g \cong i_! (f \odot
  g)$ is a cofibration in $\qsSet$, acyclic if either $f$
  or $g$ is.
  
  For the monoid axiom, first note that it is sufficient to check
  that
  \begin{equation*}
    \Cell \big\{ X \otimes \sqcap^n_{i,\varepsilon} \to X\otimes
    \square^n 
    \bigm| \text{$X\in\Ob \qsSet$,
    $1\le i\le n$ and $\varepsilon=0,1$} \big\}
  \end{equation*}
  comprises $\infty$-equivalences by \cite[Lemma 3.5 (2)]{schwedeshipley}.
  Let $\varepsilon = 0$ or $1$ and let $n > 0$.
  For arbitrary extended cubical sets $Y$, the map
  \begin{equation*}
    \id_Y \otimes \{\varepsilon\} : 
    Y \otimes \square^0_\Sigma \to Y \otimes \square^1_\Sigma
  \end{equation*}
  is a section of an $\infty$-equivalence by Lemma
  \ref{lemma:projectionequivalence}, so it is itself a weak
  equivalence.  Now consider the pushout
  \begin{equation*}
    \xymatrix@C=4pc{X \otimes \partial \square^{n-1}_\Sigma
      \ar[r]^{\id\otimes\{1-\varepsilon\}}  \ar[d]_g &
      X \otimes \partial \square^{n-1}_\Sigma \otimes \square^1_\Sigma
      \ar[d] \ar@/^/[ddr] \\
      X\otimes \square^{n-1}_\Sigma \ar[r]^k
      \ar@/_/[drr]_{\id\otimes\{1-\varepsilon\}} &
      X \otimes i_!\sqcap^n_{n,\varepsilon} \ar[dr]^\ell \\
      && X \otimes \square^n_\Sigma.}
  \end{equation*}
  By the two-out-of-three axiom and Lemma
  \ref{lemma:projectionequivalence} we know that $\ell$ is an
  $\infty$-equivalence if and only if $k$ is an $\infty$-equivalence.  But $g$
  is a monomorphism by Lemma
  \ref{lemma:extendedproductpreservesmonos}.  so $i^\ast g$ is a
  monomorphism.  Since $i^\ast k$ is the cobase change of an
  $\infty$-equivalence along a cofibration and $\qSet$ is left
  proper, $i^\ast k$ is an $\infty$-equivalence, so $k$ is an
  $\infty$-equivalence.  The cosymmetry maps allow us to permute the
  lower cap coordinate $n$.
\end{proof}

\section{Diagrams of extended cubical sets and regularity}
Recall that if $X$ is a simplicial set, there is a weak equivalence
\begin{equation*}
  \hocolim_{\Delta[n] \to X} \Delta[n] \to X
\end{equation*}
induced by the identification of $X$ with the colimit of its
simplices.  There are various ways to prove this; one method uses
Reedy model structures to show that the honest colimit of the diagram
of simplices of $X$ computes the homotopy colimit.  In this section,
we'll prove an analogous formula for (extended) cubical sets: sets:
these can be decomposed as the homotopy colimit of their cubes.

Suppose $\mathscr{R}$ is a Reedy category and $\mathscr{C}$ is a model
category.  The category $\mathscr{C}^\mathscr{R}$ of
$\mathscr{R}$-diagrams in $\mathscr{C}$ may be equipped with
\emph{Reedy model structure} \cite{hovey,hirschhorn,dhks}.  This by
now is a well-known construction; we've implicitly used it in
describing directed colimits and pushouts of weak equivalences.  We'll
give a brief overview here.
\begin{definition}[{\cite[Chapter 15]{hirschhorn}}]
  Suppose $r\in \Ob\mathscr{R}$.
  \begin{enumerate}
  \item We define $\partial (\mathscr{R}^{+} \downarrow r)$ to be the
    full subcategory of $\mathscr{R}^{+} \downarrow r$ consisting of
    non-identity arrows $s\to r$.  Let $F\in\mathscr{C}^\mathscr{R}$.
    The \emph{$r$th latching object of $F$} is the colimit
    \begin{equation*}
      L_r F = \colim_{\partial (\mathscr{R}^{+} \downarrow r)} F \in \mathscr{C}.
    \end{equation*}
    This is functorial in $F$.  Note there is a natural map $L_r F \to
    F_r$.  Suppose $f:F\to G$ is an arrow in
    $\mathscr{C}^{\mathscr{R}}$.  We say $f$ is a \emph{Reedy
      cofibration} if each corner map
    \begin{equation*}
      F_r \amalg_{L_r(F)} L_r(G) \to G_r,
    \end{equation*}
    $r\in\Ob\mathscr{R}$, is a cofibration in $\mathscr{C}$.
  \item We define $\partial (r\downarrow \mathscr{R}^{-})$ to be the
    full subcategory of $r\downarrow \mathscr{R}^{-}$ consisting of
    non-identity arrows $r\to s$.  The \emph{$r$th matching object of
      $F$} is the limit
    \begin{equation*}
      M_r F = \lim_{\partial (r\downarrow \mathscr{R}^{-})} F \in
      \mathscr{C}.
    \end{equation*}
    This is functorial in $F$ and there is a natural transformation
    $({-})_r \to M_r$.  An arrow $f:F\to G$ in $\mathscr{C}^{\mathscr{R}}$
    is a \emph{Reedy fibration} if each corner map
    \begin{equation*}
      F_r \to M_r F \times_{M_r G} G_r,
    \end{equation*}
    $r\in\Ob\mathscr{R}$, is a fibration in $\mathscr{C}$.
  \item We call a map $f:F\to G$ in $\mathscr{C}^\mathscr{R}$ an
    \emph{objectwise weak equivalence} if $f_r:F_r \to G_r$ is a
    weak equivalence for all $r\in\Ob\mathscr{R}$.
  \end{enumerate}
\end{definition}

\begin{theorem}[{\cite[Theorems 15.3.4, 15.3.15,
    15.6.27]{hirschhorn}}]\label{theorem:hirschhornreedyomnibus}
  Suppose $\mathscr{C}$ is a model category.
  \begin{enumerate}
  \item The category $\mathscr{C}^{\mathscr{R}}$ of diagrams has a
    model category structure with cofibrations the Reedy cofibrations,
    fibrations the Reedy fibrations, and weak equivalences the
    objectwise weak equivalences.
  \item If $\mathscr{C}$ is cofibrantly generated, the Reedy model
    structure on $\mathscr{C}^\mathscr{R}$ is cofibrantly generated as
    well.
  \item In $\mathscr{C}^\mathscr{R}$, an arrow $f:F\to G$ is an
    acyclic cofibration if and only if each corner map $L_r G
    \amalg_{L_r F} F_r \to G_r$ is an acyclic cofibration for all
    $r\in\Ob\mathscr{R}$.  Dually, $f$ is an acyclic fibration if and
    only if each corner map $F_r \to G_r \times_{M_r F} M_r G$ is an
    acyclic fibration.
  \end{enumerate}
\end{theorem}

Recall that if $\mathscr{R}$ is a Reedy category and
$X\in\widehat{\mathscr{R}}$, then $\mathscr{R} \downarrow X$ is a
Reedy category as well.
\begin{proposition}\label{prop:ezcolimisleftquillen}
  Suppose $\mathscr{R}$ is an \textsc{ez} Reedy category and
  $X\in\widehat{\mathscr{R}}$.  Then if $Z\in\Ob\mathscr{C}$ is
  fibrant, the constant diagram $\mathscr{R}\downarrow X \to
  \mathscr{C}$ on the object $Z$ is Reedy fibrant.
\end{proposition}
\begin{proof}
  This is a straightforward generalization of \cite[Proposition
  15.10.4]{hirschhorn}.
  Let $c_Z:\mathscr{R} \downarrow X \to \mathscr{C}$ denote the
  constant diagram on $Z$.  We need to check that for all $f:[r]\to
  X$, $Z \to M_f(c_Z)$ is a fibration in $\mathscr{C}$.  Recall that
  the $f$th matching object is computed by a limit indexed on
  $\mathscr{I} = \partial (f \downarrow (\mathscr{R}\downarrow
  X)^{-})$.  If $\mathscr{I}$ is empty, then $M_r(c_Z) = \ast$ and
  $Z\to\ast$ is a fibration by assumption.  Suppose $\mathscr{I}$ is
  nonempty.  Using the notation of Section \ref{section:qisreedy},
  suppose $[r]\to [s_i] \to X$ are two arrows in $\mathscr{I}$.  We
  may take the absolute pushout of $r\to s_1$ and $r\to s_2$ in
  $\mathscr{R}$:
  \begin{equation*}
    \xymatrix{[r] \ar[r]^{\sigma_1} \ar[d]_{\sigma_2} & [s_1]
    \ar[d]_{\tau_2} \ar@/^/[ddr] \\
    [s_2] \ar[r]^{\tau_1} \ar@/_/[rrd] & [t] \ar@{.>}[dr]^{\exists !} \\
    && X.}
  \end{equation*}
  Note that $\tau_1$ and $\tau_2$ are in $\mathscr{R}^{-}$.  Hence
  $\Nerve \mathscr{I}$ is connected.  (In fact, $\mathscr{I}$ has a terminal
  object given by the \textsc{ez} decomposition of $f$.)  Write
  $\pi:\mathscr{I}\to\ast$; the functor $\pi$ is thus left cofinal, so
  $\id \to \pi_\ast \pi^\ast$ is a natural isomorphism \cite{maclane}.
  Hence $Z\to M_r(c_Z)$ is isomorphic to the identity map on $Z$, so
  it is a fibration.
\end{proof}

Let $\partial (\mathscr{R}\downarrow r)$ denote the category of
$\mathscr{R}$-simplices $\mathscr{R} \downarrow \partial [r]$.  This
is the full subcategory of $\mathscr{R}\downarrow r$ spanned by the
objects those arrows $x\to r$ factoring through some object $s$, $\deg
s < \deg r$.
\begin{lemma}[{\cite[Proposition 15.2.8]{hirschhorn}}]\label{lemma:standardreedycoend}
  Suppose $\mathscr{R}$ is a Reedy category and $r\in\Ob\mathscr{R}$.
  The inclusion functor
  \begin{equation*}
    j:\partial(\mathscr{R}^{+} \downarrow r) \to \partial(\mathscr{R}
    \downarrow r)
  \end{equation*}
  is homotopy right cofinal.
\end{lemma}
\begin{proof}
  For (1), let $f:x \to r$ be a non-identity map in $\mathscr{R}$.
  We factor $f$ as $f = f^{+} f^{-}$, $f^{+}\in\mathscr{R}^+$,
  $f^{-}\in\mathscr{R}^-$.  Suppose
  \begin{equation*}
    \xymatrix{x \ar[rr]^f \ar[dr]_k && r \\
      & s \ar[ur]_{\ell^{+}\in\Ar\mathscr{R}^{+}}}
  \end{equation*}
  is an object in $f \downarrow j$.  We factor $k = k^{+} k^{-}$, so
  $f = (\ell^{+} k^{+}) k^{-}$.  Since $\mathscr{R}$ is Reedy, $k^{-}
  = f^{-}$ and $\ell^{+} k^{+} = f^{+}$, so the triangle
  \begin{equation*}
    \xymatrix{x \ar[rr]^f \ar[dr]_{f^{-}} && r \\
      & s' \ar[ur]_{f^{+}}}
  \end{equation*}
  is terminal in $f \downarrow j$.  Hence $\Nerve (f\downarrow j)$ is
  contractible.
\end{proof}

\begin{proposition}\label{prop:canonicaldiagramisreedycofibrant}
  Suppose $\mathscr{R}$ is an \textsc{ez} Reedy category and
  $X\in\widehat{\mathscr{R}}$.  Let $\widetilde{X}$ be the diagram
  \begin{equation*}
    \xymatrix{\mathscr{R} \downarrow X \ar[r]^\pi & \mathscr{R} \ar[r]^{[{-}]} &
    \widehat{\mathscr{R}}}
  \end{equation*}
  For an $r$-simplex $f:[r]\to X$, the Reedy map $L_f \widetilde{X}
  \to \widetilde{X}_f$ is isomorphic to the inclusion $\partial [r]
  \to [r]$.
\end{proposition}
\begin{proof}
  The forgetful functor
  \begin{equation*}
    u:\partial\big((\mathscr{R}\downarrow X) \downarrow f\big) \to 
    \partial(\mathscr{R} \downarrow r)
  \end{equation*}
  has a left adjoint sending $j:s\to r$ to
  \begin{equation*}
    \xymatrix{[s] \ar[r]^{[j]} & [r] \ar[r]^{f} & X,}
  \end{equation*}
  so $u$ is (homotopy) right cofinal and the map
  \begin{equation*}
    \colim_{\partial \big((\mathscr{R}\downarrow X) \downarrow f\big)}
    \widetilde{X} \to
    \colim_{\partial(\mathscr{R} \downarrow r)} [{-}]
  \end{equation*}
  is an isomorphism.  By Lemma \ref{lemma:standardreedycoend}, the
  Reedy map $L_f \widetilde{X} \to \widetilde{X}_f$ is thus
  isomorphic to the map
  \begin{equation*}
    \int^{s\in\Ob \mathscr{R}} \big(\partial[r]\big)(s) \times [s] 
    \to \int^{s\in\Ob \mathscr{R}} [r](s) \times [s] 
  \end{equation*}
  This is precisely the map $\partial [r] \to [r]$.
\end{proof}

\begin{corollary}\label{cor:qqisregular}
  Suppose $X\in\qSet$.  Then the natural map
  \begin{equation*}
    \hocolim_{\square^n \to X} \square^n \to X
  \end{equation*}
  is an $\infty$-equivalence.
\end{corollary}
\begin{proof}
  Recall that $\QQ$ is \textsc{ez} and Reedy.  By Proposition
  \ref{prop:ezcolimisleftquillen}, the adjoint pair
  \begin{equation*}
    \xymatrix{\colim : \qSet^{\QQ \downarrow X} \ar@<0.5ex>[r] &
    \qSet : c \ar@<0.5ex>[l]}
  \end{equation*}
  is a Quillen adjunction, so we may use the Reedy model structure on
  $\qSet^{\QQ\downarrow X}$ to compute homotopy colimits.  The
  canonical diagram taking $\square^n\to X$ to $\square^n$ is Reedy
  cofibrant by Proposition
  \ref{prop:canonicaldiagramisreedycofibrant}.  Hence
  \begin{equation*}
    \hocolim_{\square^n \to X} \square^n \to \colim_{\square^n \to X}
    \square^n \cong X
  \end{equation*}
  is an $\infty$-equivalence.
\end{proof}

Corollary \ref{cor:qqisregular} records one of the most important
properties of $\qSet$: every cubical set is the homotopy
colimit of its cubes.  Using Cisinski's terminology, the spatial model
structure on $\qSet$ is \emph{regular}.  As we'll see below,
$\qsSet$ is regular as well, but this is significantly more
difficult to prove.

\subsection{Regularity in $\qsSet$}
In the remainder of this section, we'll show that
\begin{equation*}
  \hocolim_{\square_\Sigma^n \to X} \square_\Sigma^n \to X
\end{equation*}
is an $\infty$-equivalence for all extended cubical sets $X$, i.e., that
all extended cubical sets are regular.  Our proof uses the
\emph{internal nerve} construction of Cisinski
\cite{cisinskithesis,jardinecubes}:
\begin{definition}
  Suppose $\mathscr{I}$ is a small category and $\mathscr{C}$ is a
  cofibrantly generated model category.  The \emph{internal nerve} of
  $\mathscr{I}$ in $\mathscr{C}$ at an object $X$ is the homotopy
  colimit $\hocolim_\mathscr{I} X$ of the constant diagram at $X$.  We
  denote this by $\Nerve_{\mathscr{C},X} \mathscr{I}$.
  Writing $p$ for the projection $\mathscr{I}\to\ast$, we have
  $\Nerve_{\mathscr{C},X} = \mathbf{L}p_! p^\ast X$.
  When $X$ is the terminal object $\ast$, we'll abbreviate
  $\Nerve_\mathscr{C} \mathscr{I} = \Nerve_{\mathscr{C},\ast} \mathscr{I}$.
\end{definition}
\begin{example}
  In $\sSet$, $\Nerve_{\sSet} \mathscr{I}$ is weakly
  equivalent to the nerve of $\mathscr{I}$.  Using the bar resolution,
  these are isomorphic.
\end{example}
\begin{remark}
  Internal nerve, as we've defined it, is not functorial.  What we
  have is the following: suppose $f:\mathscr{A} \to \mathscr{B}$ is a
  functor between small categories.  The triangle
  \begin{equation*}
    \xymatrix{\mathscr{A} \ar[rr]^f \ar[rd]_p && \mathscr{B} \ar[ld]^q
      \\
      & \ast}
  \end{equation*}
  yields a natural transformation $\mathbf{L}p_! p^\ast \to
  \mathbf{L}q_! q^\ast$ since $q_! f_!\cong p_!$ and $p^\ast = f^\ast
  q^\ast$.  This may be used to give $\Nerve_{\mathscr{C},X}$ the structure
  of a suitably weak $2$-functor.  We won't need that here; we'll
  write $\Nerve_\mathscr{C} f: \Nerve_\mathscr{C}\mathscr{A} \to \Nerve_{\mathscr{C}}
  \mathscr{B}$ below, but we'll be careful not to compose maps.
\end{remark}
\begin{proposition}\label{prop:internalnervepreservesthomasoneqs}
  Suppose $f:\mathscr{A}\to\mathscr{B}$ is a functor between small
  categories.  Then $\Nerve_{\qSet}f$ and $\Nerve_{\qsSet} f$
  are $\infty$-equivalences if and only if $f$ is a Thomason equivalence.
\end{proposition}
\begin{proof}
  In $\qSet$, $|{-}|$ coincides with its left derived functors
  as everything is cofibrant.  This may not be the case in
  $\qsSet$. However,
  $|{-}|^\mathbf{L}_{\qsSet}$ preserves and reflects
  weak equivalences, where $|{-}|^\mathbf{L}_{\qsSet}$
  denotes the left derived functor of extended realization.  We have squares
  \begin{equation*}
    \vcenter{\xymatrix@C=3pc{\Nerve_{\sSet} \mathscr{A}
      \ar[r]^{\Nerve_{\sSet} f} \ar[d]_\sim & \Nerve_{\sSet} \mathscr{B}
      \ar[d]^\sim \\
      |\Nerve_{\qSet} \mathscr{A}|
      \ar[r]_{|\Nerve_{\qSet} f|} & |\Nerve_{\qSet}
      \mathscr{B}|}}
    \qquad\text{and}\qquad
    \vcenter{\xymatrix@C=3pc{\Nerve_{\sSet} \mathscr{A}
      \ar[r]^{\Nerve_{\sSet} f} \ar[d]_\sim & \Nerve_{\sSet} \mathscr{B}
      \ar[d]^\sim \\
      |\Nerve_{\qsSet} \mathscr{A}|_\Sigma^\mathbf{L}
      \ar[r]_{|\Nerve_{\qsSet} f|_\Sigma^\mathbf{L}} &
      |\Nerve_{\qsSet} \mathscr{B}|_\Sigma^\mathbf{L}}}
  \end{equation*}
  commuting up to natural weak equivalence, so $\Nerve_{\qSet}f$
  and $\Nerve_{\qsSet}f$ are $\infty$-equivalences if and only
  if $f$ is a Thomason equivalence.
\end{proof}
\begin{remark}
  Proposition \ref{prop:internalnervepreservesthomasoneqs} is part of
  a general yoga of categorical homotopy theory due to Cisinski
  \cite{cisinskithesis,jardinecubes}: the homotopy theory of
  categories (i.e., spaces) intervenes in every model category via the
  internal nerve.
\end{remark}
\begin{proposition}\label{prop:qsigmaisregular}
  Suppose $X$ is an extended cubical set.  The natural map
  \begin{equation*}
    \hocolim_{\square^n_\Sigma \to X} \square^n_\Sigma \to X
  \end{equation*}
  is a natural $\infty$-equivalence.
\end{proposition}
\begin{proof}
  By Corollary \ref{cor:qqisregular}, the map
  \begin{equation*}
    \hocolim_{\square^n \to i^\ast X} \square^n \to i^\ast X
  \end{equation*}
  is an $\infty$-equivalence in $\qSet$.  Since $i_!$ is left
  Quillen and all cubical sets are cofibrant, the map
  \begin{equation*}
    \hocolim_{\square^n \to i^\ast X \in \QQ\downarrow i^\ast X} \square_\Sigma^n \to i_! i^\ast X
  \end{equation*}
  is an $\infty$-equivalence in $\qsSet$.  Let $G$ denote
  the canonical diagram of cubes of $X$:
  \begin{equation*}
    \xymatrix{\qsigma \downarrow X \ar[r]^\pi & \qsigma \ar[r]^r & \qsSet.}
  \end{equation*}
  Recall that $i$ induces a functor $j:\QQ\downarrow i^\ast X \to
  \qsigma \downarrow X$ and that $j$ is a Thomason equivalence by
  Proposition \ref{prop:iastreflectsweakequivalences}.  Note that $F =
  G j$ is roughly the diagram of cubes of $i^\ast X$: it is the
  functor
  \begin{equation*}
    \xymatrix{\QQ \downarrow i^\ast X \ar[r]^\pi & \QQ \ar[r]^r &
      \qSet \ar[r]^{i_!} & \qsSet.}
  \end{equation*}
  The natural transformation $\mathbf{L}j_! j^\ast \to \id$ induces the left arrow in
  \begin{equation*}
    \xymatrix{\hocolim_{\QQ\downarrow i^\ast X} F \ar[r] \ar[d] &
      \colim_{\QQ\downarrow i^\ast X} F \ar@{=}[r] \ar[d] & i_! i^\ast X
      \ar[d]^{\varepsilon_X} \\
      \hocolim_{\qsigma\downarrow X} G \ar[r] &
      \colim_{\qsigma\downarrow X} G \ar@{=}[r] & X,}
  \end{equation*}
  which commutes up to natural $\infty$-equivalence.  Thus it is
  sufficient to show that
  \begin{equation*}
    \hocolim_{\QQ\downarrow i^\ast X} F \to
    \hocolim_{\qsigma\downarrow X} G
  \end{equation*}
  is an $\infty$-equivalence.  Let $\ast$ denote the constant diagram
  on the terminal object in $\qsSet$; then 
  \begin{equation*}
    \xymatrix{\hocolim_{\QQ\downarrow i^\ast X} F \ar[d] \ar[r] &
      \hocolim_{\QQ\downarrow i^\ast X} \ast \ar[d] \\
      \hocolim_{\qsigma\downarrow X} G \ar[r] &
      \hocolim_{\qsigma\downarrow X} \ast}
  \end{equation*}
  commutes up to natural $\infty$-equivalence.  The horizontal arrows are
  $\infty$-equivalences since $\hocolim$ is a homotopy functor and
  $\square^n_\Sigma\to\ast$ is an $\infty$-equivalence.  The right
  vertical arrow is $\Nerve_{\qsSet} j$; this is an
  $\infty$-equivalence by Proposition \ref{prop:internalnervepreservesthomasoneqs}.
\end{proof}

\part{Extended cubical enrichments}\label{part:two}

\section{Enriched model categories}\label{section:enriched}
Suppose $(\mathscr{V},\otimes,e)$ is a closed symmetric monoidal model
category.  We assume that the monoidal structure in $\mathscr{V}$ is
compatible with the model structure by requiring the usual axiom: the
product $\otimes$ to be a \emph{left Quillen bifunctor}, i.e., if
$k:A\to B$ and $\ell:X\to Y$ are cofibrations in $\mathscr{V}$, then the
pushout-product
\begin{equation*}
  k\odot \ell = A\otimes Y \amalg_{A\otimes X} B\otimes X \to B\otimes Y
\end{equation*}
is a cofibration, acyclic if either $k$ or $\ell$ is.
We have the following fundamental definition \cite{hovey,barwickenriched}:
\begin{definition}
  Suppose $\mathscr{C}$ is a category enriched over $\mathscr{V}$.
  Write $\mathscr{C}_0$ for the underlying $\Cat{Set}$-category of
  $\mathscr{C}$.  We say $\mathscr{C}$ is a \emph{$\mathscr{V}$-model
    category} if
  \begin{enumerate}
  \item[($\mathscr{V}$M1)] $\mathscr{C}_0$ is a model category.
  \item[($\mathscr{V}$M2)] $\mathscr{C}$ has all $\mathscr{V}$-indexed limits and
    colimits \cite{kelly}.
  \item[($\mathscr{V}$M3)] The tensor functor
    ${-}\otimes{-}:\mathscr{V}\otimes\mathscr{C} \to \mathscr{C}$ is a
    left Quillen bifunctor.
  \end{enumerate}
\end{definition}
There are several standard reductions of $\mathscr{V}$M3: the existence of
$\mathscr{V}$-indexed limits and colimits grants adjunctions
\begin{equation*}
  \mathscr{C}(A\otimes X, Y) \cong \mathscr{V}(A,
  \mathscr{C}(X,Y)) \cong \mathscr{C}(X, Y^A)
\end{equation*}
where $A\in\mathscr{V}$ and $X,Y\in\mathscr{C}$.  For example, we can
replace $\mathscr{V}$M3 with the axiom that $\mathscr{C}({-},{-})$ be a right
Quillen bifunctor, i.e, if $k$ is a cofibration and $f$ a fibration,
$\mathscr{C}(k,f)$ is a fibration, acyclic if either $k$ or $f$ is; in
the case of cofibrant generation, we need only check axiom $\mathscr{V}$M3 for
generating (acyclic) cofibrations.

Suppose $\mathscr{C}$ is a model category with a functorial ``cylinder
object,'' i.e., for every $X$, a natural factorization of the fold map
\begin{equation}
  \vcenter{\xymatrix@C=4pc{
      X\amalg X \ar[r]^{d_0\amalg d_1} \ar@/_/[rr]_{\id\amalg\id} & \Cyl(X) \ar[r] & X
    }}
\end{equation}
into a cofibration followed by a weak equivalence.  Then $\mathscr{C}$
is naturally enriched over $\qSet$ by setting
$\mathscr{C}(X,Y)_n = \Hom_{\mathscr{C}}(\Cyl^n(X),Y)$
\cite{cisinskithesis,jardinecubes}.  Dually, if $X$ has a natural
``path object''---a factorization
\begin{equation}
  \vcenter{\xymatrix@C=4pc{
      X \ar[r] \ar@/_/[rr]_{\Delta} & PX \ar[r] & X\times X
    }}
\end{equation}
of the diagonal map into a weak equivalence followed by a
fibration---we may define a cubical mapping complex
$\mathscr{C}(X,Y)_n = \Hom_{\mathscr{C}}(X, P^n Y)$.  As we'll discuss
in Section \ref{section:mappingspaces}, the cubical realization
$|\mathscr{C}(X,Y)|$ is a model for the Dwyer-Kan mapping space
between $X$ and $Y$.  However, the monoidal structure on
$\qSet$ is not symmetric, so $\mathscr{C}$ cannot
possibly be a $\qSet$-model category in the sense we described
above.  We might try to remedy this by upgrading the enrichment of
$\mathscr{C}$ from a $\qSet$-category to a
$\qsSet$-category.  This isn't possible in general, but we
have the following principle:
\begin{theorem}\label{theorem:extendedcubicalmodelcats}
  Let $\mathscr{C}$ be a symmetric monoidal model category.
  Suppose $\mathscr{C}$ possesses a cubical monoid
  \begin{equation*}
    \vcenter{\xymatrix{
        e \amalg e \ar[r]^{d_0 \amalg d_1} \ar@/_/[rr]_{\id\amalg\id} &
        I \ar[r]^s & e
      }}
  \end{equation*}
  so that $d_0\amalg d_1$ is a cofibration and $s$ a weak equivalence.
  Then $\mathscr{C}$ is a $\qsSet$-model category, with
  $\mathscr{C}(X,Y)_n = \Hom_\mathscr{C}(\sqparen{n} \otimes X, Y)$.
  Moreover, the monoidal structure on $\mathscr{C}$ is given by
  $\qsSet$-functors.
\end{theorem}
One example is furnished by $\Ch(R)$, $R$ a commutative ring: the
normalized $R$-chains of $\Delta[1]$ give a cubical monoid with the
appropriate homotopical properties.  This Theorem amounts to the fact
that the $\mathscr{C}$-algebras of the \textsc{prop} $\qsigma$ are
precisely cubical monoids.  In the remainder of this section, we'll
show that the $\qsSet$ mapping spaces given by Theorem
\ref{theorem:extendedcubicalmodelcats} have the correct homotopy type
(i.e.~the homotopy type of the Dwyer-Kan mapping space) and that every
combinatorial symmetric monoidal model category with cofibrant unit
has an extended cubical enrichment.

\section{Virtual cofibrance and diagram categories}
As in Section \ref{section:enriched}, let
$(\mathscr{V},\otimes,e,[{-},{-}])$ be a closed symmetric monoidal
model category.  The fundamental example of a $\mathscr{V}$-model
category is $\mathscr{V}$ itself.  In order to discuss
$\mathscr{V}$-diagram categories, we need to introduce some technical
model categorical material first.
\begin{definition}
  Suppose $\mathscr{C}$ is a $\mathscr{V}$-model category.  We say an
  arrow $k\in\Ar\mathscr{V}$ is a \emph{$\mathscr{C}$-virtual cofibration} if
  $k\odot f$ is an (acyclic) cofibration for all (acyclic)
  cofibrations $f$ in $\Ar\mathscr{C}$.  We say $k$ is a \emph{virtual
  cofibration} if it is a $\mathscr{C}$-virtual cofibration for all
  $\mathscr{V}$-model categories $\mathscr{C}$.
\end{definition}
The following Proposition is straightforward:
\begin{proposition}
  \begin{enumerate}
  \item All cofibrations are virtual cofibrations.
  \item The class of virtual cofibrations in $\mathscr{C}$ is closed
    under coproduct, cobase change, transfinite composition, and
    retract.
  \item Virtual (acyclic) cofibrations and (acyclic) cofibrations
    coincide if and only if $\emptyset\to e$ is a cofibration in
    $\mathscr{V}$.
  \end{enumerate}
\end{proposition}
Note that $\emptyset\to e$ is always a virtual cofibration, but it
need not be a cofibration.
\begin{definition}
  Suppose $\mathscr{I}$ is a small $\mathscr{V}$-category.  We say
  $\mathscr{I}$ has \emph{virtually cofibrant mapping spaces} if
  $\emptyset\to\mathscr{I}(x,y)$ is a cofibration for all
  $x,y\in\mathscr{I}$.  If, furthermore, $e\to\mathscr{I}(x,x)$ is a
  cofibration for all $x\in\mathscr{I}$, we say $\mathscr{I}$ is
  \emph{well based}.
\end{definition}

\begin{proposition}\label{prop:enrichedomnibus}
  Suppose $\mathscr{V}$ is combinatorial (i.e., cofibrantly generated
  and locally presentable; see
  \cite{dugger1,adamekrosicky,makkaipare}).  
  Let $\mathscr{I}$ is a small $\mathscr{V}$-category with
  virtually cofibrant mapping spaces, e.g., $\mathscr{I}$ is the free
  $\mathscr{V}$-category generated by a $\Cat{Set}$-category.
  \begin{enumerate}
  \item The $\mathscr{V}$-category $\widehat{\mathscr{I}} =
    \mathscr{V}^{\mathscr{I}^\Op}$ admits a cofibrantly generated
    model structure, known as the \emph{projective model structure},
    in which $f:X\to Y$ is a fibration (respectively weak equivalence)
    if and only if $f_i:X_i \to Y_i$ is a fibration (respectively weak
    equivalence) for all $i\in\mathscr{I}$.
  \item Suppose further that $\mathscr{I}$ is symmetric monoidal; then
    the category $\widehat{\mathscr{I}}$ admits a closed symmetric
    monoidal structure given by Day convolution.  The category
    $\widehat{\mathscr{I}}$ with the projective model structure is
    then a symmetric monoidal model category.  If $\mathscr{V}$
    satisfies the Schwede-Shipley monoid axiom, then so does
    $\widehat{\mathscr{I}}$.
  \end{enumerate}
\end{proposition}
\begin{proof}
  For 1, note that if $K\to L$ is an acyclic cofibration in
  $\mathscr{V}$ and $i\in\mathscr{I}$, the left Kan extension
  \begin{equation*}
    \mathscr{I}({-},x) \otimes K \to \mathscr{I}({-},x) \otimes L
  \end{equation*}
  must be a weak equivalence in the projective model structure on
  $\widehat{\mathscr{I}}$ (indeed, an acyclic cofibration).  The virtual
  cofibrance assumption for $\mathscr{I}$ guarantees this.
  The combinatoriality of $\mathscr{V}$ ensures that we may run the
  small-object argument.

  For 2, suppose that $\mathscr{I}$ is symmetric monoidal.  Since the
  convolution product makes Yoneda strong monoidal, and the product on
  $\mathscr{V}$ preserves colimits in each variable,
  for arbitrary arrows $K\to L$, $A\to B$ in $\mathscr{V}_0$ and 
  objects $x,y \in\mathscr{I}$, we have
  \begin{align*}
    &\big(\mathscr{I}({-},x) \otimes K \to \mathscr{I}({-},x) \otimes
    L\big)
    \odot
    \big(\mathscr{I}({-},y) \otimes A \to \mathscr{I}({-},y) \otimes
    B\big) \\ &\qquad\cong
    \mathscr{I}({-},x\otimes y) \otimes \big( K\otimes
    B\amalg_{K\otimes A} L\otimes A \to L\otimes B\big).
  \end{align*}
  As a result, the monoidalness of the model structure on
  $\mathscr{V}$ lifts to show that $\widehat{\mathscr{I}}$ is a
  monoidal model category.

  To show that $\widehat{\mathscr{I}}$ satisfies the Schwede-Shipley
  monoid axiom, it is sufficient to check that the arrows in
  \begin{equation*}
    \Cell \big\{ \mathscr{I}({-},x)\otimes F \otimes k \bigm|
    \text{$x\in\mathscr{I}$, $F\in\widehat{\mathscr{I}}$ and $k$ an
    acyclic cofibration} \big\}
  \end{equation*}
  are weak equivalences in $\widehat{\mathscr{I}}$.  Since cobase
  change, transfinite composition, retract and coproduct all commute
  with the evaluation functors $\widehat{\mathscr{I}} \to \mathscr{V}$
  and $\widehat{\mathscr{I}}$ has the projective model structure, it
  is sufficient to check that
  \begin{equation*}
    \Cell \big\{ \big(\mathscr{I}({-},x)\otimes F\big)_z \otimes k \bigm|
    \text{$x\in\mathscr{I}$, $F\in\widehat{\mathscr{I}}$ and $k$ an
    acyclic cofibration} \big\}
  \end{equation*}
  consists of weak equivalences in $\mathscr{V}$
  for all $z\in \mathscr{I}$.  This is guaranteed by the monoid axiom
  for $\mathscr{V}$.
\end{proof}

\section{Cubical models for mapping spaces}\label{section:mappingspaces}
In the series of papers \cite{dwyerkan1,dwyerkan2,dwyerkan3}, Dwyer
and Kan introduced the \emph{simplicial localization} of a category at
a subcategory of weak equivalences.  In the case of a model category
$\mathscr{C}$, the simplicial localization of $\mathscr{C}$ at its
weak equivalences $\mathscr{W}$ associates a simplicial set
$\mathbf{F}(x,y)$ to each pair of objects $x,y$ so that
$\pi_0\mathbf{F}(x,y)$ corresponds to the set $\Ho\mathscr{C}(x,y)$ in
a natural way.  When $\mathscr{C}$ is a simplicial model category and
$\mathbf{F}(x,y)$ is the derived mapping space $\mathscr{C}(x', y')$
($x'$ and $y'$ are cofibrant-fibrant replacements for $x$ and $y$,
respectively).  In Section \ref{section:enriched}, we discussed
cubical enrichments of model categories (following Cisinski).  In this
section, we'll show that those enrichments have the correct homotopy
type, i.e., coincide with the space $\mathbf{F}(x,y)$ up to weak
equivalence.

\subsection{Quillen adjunctions between Reedy categories}
The fundamental technical tool we'll need is a comparison between
cubical and simplicial framings.  Conversion between the two is
essentially obtained by cubical realization.  Most of the material in
this section has a straightforward generalization to enriched
categories; we won't need that here.  Recall that a small-cocomplete
and small-complete category $\mathscr{C}$ is tensored and cotensored
over $\Cat{Set}$ by the copower and power operations: there are
adjunctions
\begin{equation}\label{eq:chassettensors}
  \mathscr{C}(S\times X, Y) \cong \Cat{Set}(S, \mathscr{C}(X, Y))
  \cong \mathscr{C}(X, Y^S)
\end{equation}
where $S\in\Ob\Cat{Set}$, $X,Y\in\Ob\mathscr{C}$.
\begin{definition}
  Suppose $\mathscr{A}$ and $\mathscr{B}$ are small categories.  A
  \emph{distributor from $\mathscr{A}$ to $\mathscr{B})$} is a functor
  \begin{equation*}
    K:\mathscr{A} \times \mathscr{B}^\Op \to \Cat{Set}.
  \end{equation*}
\end{definition}
This definition is due to Benabou (see \cite{cordierporter}).  We
sometimes denote $K$ by a dashed arrow $\mathscr{A} \dashrightarrow
\mathscr{B}$.  The data of a distributor $K:\mathscr{A}
\dashrightarrow \mathscr{B}$, via the universal property of the
Yoneda embedding, is equivalent to an adjunction
\begin{equation*}
  \xymatrix{L_K : \widehat{\mathscr{A}} \ar@<0.5ex>[r] &
    \widehat{\mathscr{B}} : R_K. \ar@<0.5ex>[l]}
\end{equation*}
Now suppose $\mathscr{C}$ has all small colimits and limits and
$K:\mathscr{A}\dashrightarrow\mathscr{B}$ is a distributor.  Given
diagrams $X\in\mathscr{C}^{\mathscr{A}^\Op}$ and
$Y\in\mathscr{C}^{\mathscr{B}^\Op}$, we define
\begin{equation*}
  \big(L_{K,\mathscr{C}}X\big)_b = \int^{a\in\Ob\mathscr{A}} K^a_b\times X_a
  \qquad\text{and}\qquad
  \big(R_{K,\mathscr{C}}Y\big)_a = \int_{b\in\Ob\mathscr{B}} (Y_b)^{K^a_b}.
\end{equation*}
\begin{proposition}\label{prop:kerneladjunction}
  Suppose $K:\mathscr{A}\dashrightarrow\mathscr{B}$ is a distributor.
  If $\mathscr{C}$ is small-cocomplete and complete, then
  \begin{equation*}
    \xymatrix{L_{K,\mathscr{C}} : \mathscr{C}^{\mathscr{A}^\Op} \ar@<0.5ex>[r] &
       \mathscr{C}^{\mathscr{B}^\Op} : R_{K,\mathscr{C}} \ar@<0.5ex>[l]}
  \end{equation*}
  is an adjoint pair.
\end{proposition}
This boils down to \eqref{eq:chassettensors}.  Note that Proposition
\ref{prop:kerneladjunction} does not have a converse in general---not
all adjunctions between $\mathscr{C}^{\mathscr{A}^\Op}$ and
$\mathscr{C}^{\mathscr{B}^\Op}$ are given by distributors.

Suppose $\mathscr{I}$ is a small category.  The copower operation
induces a bifunctor
\begin{gather*}
  {-}\otimes{-} : \widehat{\mathscr{I}} \times \mathscr{C} \to
  \mathscr{C}^{\mathscr{I}^\Op}
\end{gather*}
for small $\mathscr{I}$ with $(A \otimes X)_a = A_a\times X$.  These
functor is divisible on both sides \cite{joyaltierney}: abusing
notation a bit, there are adjunctions
\begin{equation*}
  \mathscr{C}(X, Y^A) \cong \mathscr{C}^{\mathscr{I}^\Op}(A\otimes X,
  Y) \cong \widehat{\mathscr{I}}(A, [X, Y])
\end{equation*}
with
\begin{equation*}
  Y^A = \int_{a\in\mathscr{I}} (Y_a)^{A_a}\qquad\text{and}\qquad
  [X,Y]_a = \mathscr{C}(X, Y_a).
\end{equation*}
In fact, the $\otimes$ bifunctor is really part of an
$\widehat{\mathscr{I}}$-enrichment on $\mathscr{C}^{\mathscr{I}^\Op}$
known as the \emph{external $\widehat{\mathscr{I}}$-enrichment}.
Since we won't need the full power of the external enrichment, we've
only defined $A\otimes{-}$ for constant diagrams in
$\mathscr{C}^{\mathscr{I}^\Op}$ and we've taken global sections in
defining its right adjoint ${-}^A$.  Note that if $[a]$ is the
presheaf represented by $a$, then $Y^{[a]}$ is naturally isomorphic to
$Y_a$.  The pushout-product
\begin{equation*}
  {-}\odot{-}:\Ar\widehat{\mathscr{I}} \times \Ar\mathscr{C} \to \Ar\mathscr{C}^{\mathscr{I}^\Op}
\end{equation*}
has adjoints
\begin{equation*}
  \Ar \mathscr{C}(f, \langle i \backslash g\rangle) \cong
  \Ar \mathscr{C}^{\mathscr{I}^\Op}(i \odot f, g) \cong
  \Ar \widehat{\mathscr{I}}(i, \langle g /f \rangle).
\end{equation*}
Now suppose $\mathscr{I}$ is an \textsc{ez} Reedy category and
$\mathscr{C}$ is a model category.  (In the remainder of this
paragraph, the \textsc{ez} assumption can be weakened by
redefining $\partial[i]$; see \cite{hirschhorn,barwickreedy}.)  For
$a\in\Ob\mathscr{I}$, recall we defined presheaves $[a]$ and
$\partial[a]$ in Section \ref{section:qisreedy}; in the case of $\QQ$,
$\square^n = [\sqparen{n}]$ and $\partial\square^n = \partial[\sqparen{n}]$.  Given
$Y\in\mathscr{C}^{\mathscr{I}^\Op}$, we may realize the Reedy matching
object $M_a Y$ as $Y^{\partial [a]}$.  A map $g:Y\to Z$ in
$\mathscr{C}^{\mathscr{I}^\Op}$ is thus a Reedy fibration if and only
if
\begin{equation*}
  \langle i_a \backslash g\rangle : Y^{[a]}
  \to Y^{\partial [a]} \times_{Z^{\partial [a]}} Z^{[a]},\ i_a =
  \partial [a] \to [a]
\end{equation*}
is a fibration in $\mathscr{C}$ for all $a\in\Ob\mathscr{I}$.  What's
not obvious, but still true, is that $g$ is a Reedy acyclic fibration
if and only if $\langle i_a \backslash g\rangle$ is an acyclic
fibration for all $a\in\Ob\mathscr{A}$---recall that weak equivalences
in the Reedy model structure are defined objectwise, not in terms of
mapping or latching objects.  Hirschhorn proves this result in
\cite[Theorem 15.3.15]{hirschhorn}; we quoted it earlier as Theorem
\ref{theorem:hirschhornreedyomnibus}.  We'll use it to compare Reedy
model structures in Proposition \ref{prop:reedyinheritsleftquillen}
below.
\begin{lemma}\label{lemma:matchingobjectcomputation}
  Let $\mathscr{A}$ and $\mathscr{B}$ be small categories and
  $K:\mathscr{A}\dashrightarrow\mathscr{B}$ a distributor.  Suppose
  $\mathscr{C}$ has all small colimits and limits.  There is an
  isomorphism $(R_{K,\mathscr{C}} Y)^A \cong Y^{L_K A}$ natural in
  $A\in\Ob\widehat{\mathscr{A}}$ and
  $Y\in\Ob\mathscr{C}^{\mathscr{B}^\Op}$.
\end{lemma}
This is an exercise in adjunctions.

\begin{proposition}\label{prop:reedyinheritsleftquillen}
  Suppose $\mathscr{C}$ is a model category and both $\mathscr{A}$ and
  $\mathscr{B}$ are \textsc{ez} Reedy categories.  Suppose
  $K:\mathscr{A}\dashrightarrow\mathscr{B}$ is a distributor so that
  $L_K:\widehat{\mathscr{A}} \to \widehat{\mathscr{B}}$ preserves
  monomorphisms.  Then
  \begin{equation*}
    \xymatrix{L_{K,\mathscr{C}} : \mathscr{C}^{\mathscr{A}^\Op} \ar@<0.5ex>[r] &
       \mathscr{C}^{\mathscr{B}^\Op} : R_{K,\mathscr{C}} \ar@<0.5ex>[l]}
  \end{equation*}
  is a Quillen pair when each category is equipped with the Reedy
  model structure.
\end{proposition}
\begin{proof}
  We'll check that $R_{K,\mathscr{C}}$ is right Quillen.  Suppose
  $p:Y\to Z$ is a Reedy fibration in $\mathscr{C}^{\mathscr{B}^\Op}$.
  By the computation in Lemma \ref{lemma:matchingobjectcomputation},
  we need $\langle L_K i_a \backslash p \rangle$ to be a fibration for
  all $a$, acyclic if $p$ is acyclic (as above, $i_a$ is the inclusion
  $\partial[a] \to [a]$).  Suppose $f:A \to B$ is a cofibration in
  $\mathscr{C}$ and one of $p$, $f$ is a weak equivalence.  It is
  sufficient to check that $f\pitchfork \langle L_K i_a \backslash
  p\rangle$, but this occurs if and only if $(L_K i_a \odot f)
  \pitchfork p$, which occurs in turn if and only if $L_K i_a
  \pitchfork \langle p/f\rangle$.  Since $p$ is a Reedy fibration, $f
  \pitchfork \langle i_b \backslash p\rangle$ for all
  $b\in\mathscr{B}$.
  Thus $i_b \pitchfork \langle p/f\rangle$ for all $b$.  By
  Corollary \ref{cor:cellularmodelforez}, the collection of maps
  $i_b$ form a cellular model for $\widehat{\mathscr{B}}$: we have
  \begin{equation*}
    \Cell \{ i_b \mid b\in\Ob \mathscr{B} \} = \mathsf{mono}_{\widehat{\mathscr{B}}}.
  \end{equation*}
  Thus $\mathsf{mono}_{\widehat{\mathscr{B}}} \pitchfork \langle p/f\rangle$.  
  By assumption, $L_K i_a$ is a monomorphism, so $L_K i_a \pitchfork
  \langle p/f\rangle$.
\end{proof}

\subsection{Cubical and simplicial resolutions}
We are now in a position to use the machinery we developed in the last
section with the distributor $K:\QQ\dashrightarrow\Delta$ associated
to cubical realization; this is given by
\begin{equation*}
  K^{\sqparen{n}}_{[m]} = \sSet(\Delta[m], |\square^n|) \cong
  \big(\sSet (\Delta[m], \Delta[1])\big)^n.
\end{equation*}
For an arbitrary model category $\mathscr{C}$, Proposition
\ref{prop:reedyinheritsleftquillen} yields a Quillen adjunction
between cubical and simplicial objects in $\mathscr{C}$:
\begin{equation*}
  \xymatrix{L_{K,\mathscr{C}} : q\mathscr{C} \ar@<0.5ex>[r] &
    s\mathscr{C} : R_{K,\mathscr{C}}. \ar@<0.5ex>[l]}
\end{equation*}

\begin{definition}[{\cite[Definition 16.1.2]{hirschhorn}}]
  Suppose $\mathscr{C}$ is a model category and $X\in\mathscr{C}$ is
  fibrant.  Write $\CS X$ and $\CQ X$ for the constant simplicial and
  cubical diagrams, respectively, with value $X$ in $\mathscr{C}$.
  A \emph{simplicial resolution} (resp.~\emph{cubical resolution}) of
  $X$ is a weak equivalence $\CS X \to \overline{X}$ (resp.~$\CQ X
  \to \widetilde{X}$) in which $\overline{X}$ (resp.~$\widetilde{X}$)
  is Reedy fibrant.
\end{definition}

We'll need this quick geometric lemma later:
\begin{lemma}\label{lemma:rightadjointpreservesconstantdiagrams}
  Let $\mathscr{C}$ be a small-complete and small-cocomplete category.
  Suppose $\CS Y$ is the constant diagram $\Delta^\Op \to \mathscr{C}$ on
  $Y$.  Then $R_{K,\mathscr{C}} \CS Y$ is isomorphic to the constant
  diagram $\CQ Y:\QQ^\Op\to\mathscr{C}$ on $Y$.
\end{lemma}
\begin{proof}
  This amounts to a verification that the representable functors
  $\square^n$ have connected geometric realization.  First, fix
  $\sqparen{n} \in \Ob\QQ$.  We have a functor $K^{\sqparen{n}}_{-}
  :\Delta^\Op \to \Cat{Set}$.  Suppose that this takes values in
  discrete categories instead and form the Grothendieck construction
  ${C_n}^\Op = \Delta^\Op \int K^{\sqparen{n}}_{-}$.  The category
  $C_n$ is isomorphic to the category $\Delta \downarrow | \square^n
  |$ of simplices of $|\square^n|$.  There is a zig-zag of weak
  equivalences from $\Nerve C_n$ to $|\square^n|$, so $C_n$ has a
  weakly contractible nerve.  Note that the $C_n$ taken together form
  a cocubical object $\QQ\to\Cat{Cat}$.

  Recall that $R_{K,\mathscr{C}} \CS Y$ is the functor
  \begin{equation*}
    R_{K,\mathscr{C}} Y(\sqparen{n}) = \int_{[m]\in\Delta} Y^{K^{\sqparen{n}}_{[m]}} \cong
    \lim_{[m]\in\Delta} Y^{K^{\sqparen{n}}_{[m]}}.
  \end{equation*}
  Let $\pi:C_n \to \Delta$ denote projection and write
  $c_Y:C_n \to \mathscr{C}$ for the constant functor on $Y$.
  The chain of functors
  \begin{equation*}
    \xymatrix{C_n \ar[r]^\pi & \Delta \ar[r] & \ast}
  \end{equation*}
  induces an isomorphism $\lim_{C_n} c_Y \cong \lim_\Delta \pi_\ast
  c_Y$, where $\pi_\ast:\mathscr{C}^{C_n}\to\mathscr{C}^\Delta$ is
  right Kan extension along $\pi$.  Suppose $[m]\in\Delta$.  Viewing
  $K^{\sqparen{n}}_{[m]}$ as a discrete category, there is a functor
  $\iota:K^{\sqparen{n}}_{[m]} \to [m]\downarrow \pi$ practically by
  definition sending the simplex $f:\Delta[m] \to |\square^n|$ to the
  pair $(f, \id_{[m]})$.  The functor $\iota$ has a right adjoint
  sending the solid arrows in the diagram
  \begin{equation*}
    \xymatrix{\Delta[m] \ar@{.>}[rd] \ar[rr] && \Delta[r] \ar[ld] \\
      & |\square^n|}
  \end{equation*}
  to the (uniquely determined) dotted arrow; the entire triangle
  displays the counit of the adjunction.  Thus $\iota$ is left
  cofinal \cite{maclane}, so there is a chain of isomorphisms
  \begin{equation*}
    (\pi_\ast c_Y)[m] \cong \lim_{[m]\downarrow \pi} c_Y \cong
    \lim_{K^{\sqparen{n}}_{[m]}} \iota^\ast c_Y \cong Y^{K^{\sqparen{n}}_{[m]}}
  \end{equation*}
  natural in $Y$, $[m]$, and $\sqparen{n}$.  Hence there is an isomorphism
  \begin{equation*}
    \lim_{[m]\in\Delta} Y^{K^{\sqparen{n}}_{[m]}} \cong \lim_{C_n} c_Y
  \end{equation*}
  natural in $\sqparen{n}$ and $Y$.  Since $\Nerve C_n$ is contractible, it is
  nonempty and connected, so the latter limit is simply $Y$.
\end{proof}

The Reedy model structure on $s\mathscr{C}$ is not
compatible with the external simplicial enrichment: if $Y$ is a Reedy
fibrant simplicial object in $\mathscr{C}$, then $Y^{\Delta[n]} \to
Y^{\Lambda_i[n]}$ need not be an objectwise weak equivalence.  However,
if $Y$ is a homotopically constant Reedy fibrant diagram, then as
we'll see below, $Y^{\Delta[n]} \to Y^{\Lambda_i[n]}$ is a weak
equivalence.  One (circular) way to think about this is that the
homotopically constant Reedy fibrant diagrams comprise the fibrant
objects in a Bousfield localization of $s\mathscr{C}$ that
is both compatible with the simplicial enrichment and Quillen
equivalent to $\mathscr{C}$ \cite{duggerreplacing,rss}.
\begin{proposition}[{\cite[Theorem 16.5.7]{hirschhorn}}]
  \label{prop:homotopicallyconstant}
  Suppose $\widetilde{Y}$ is a Reedy fibrant diagram in
  $s\mathscr{C}$.
  \begin{enumerate}
  \item If $A\to B$ is a monomorphism of simplicial sets, then $\widetilde{Y}^B
    \to \widetilde{Y}^A$ is a fibration in $\mathscr{C}$.
  \item Suppose $\widetilde{Y}$ is \emph{homotopically constant},
    i.e., that each map $\Delta[n] \to \Delta[m]$ in  $\Delta$ induces a
    weak equivalence $\widetilde{Y}_m \to \widetilde{Y}_n$.  If $A\to
    B$ is an acyclic cofibration of simplicial sets, then
    $\widetilde{Y}^B \to \widetilde{Y}^A$ is an acyclic
    fibration.
  \end{enumerate}
\end{proposition}
\begin{remark}
  Proposition \ref{prop:homotopicallyconstant} is also true for
  homotopically constant Reedy fibrant diagrams in
  $q\mathscr{C}$.
\end{remark}

We'll take the following definition as a sort of black box: given $X$
cofibrant and $Y$ fibrant in $\mathscr{C}$, the Dwyer-Kan mapping
space can be constructed by the following process: we construct a
simplicial resolution $\CS Y \to \widetilde{Y}$ and define
$\mathbf{F}(X, Y)$ to be the simplicial set $[X, \widetilde{Y}]$.
(Recall from the previous section that $[X,Y]_n = \mathscr{C}(X,
Y_n)$.  Hirschhorn shows in \cite{hirschhorn} that this is
well-defined up to weak equivalence and that it has the appropriate
functorial properties.  This construction is by no means the only way
of getting at the homotopy type of $\mathbf{F}(X, Y)$.  The following
lemma, in a simplicial guise, is found in \cite[Proposition
16.1.17]{hirschhorn}.
\begin{lemma}\label{lemma:reedyresolutions}
  Suppose $X\in\Ob\mathscr{C}$ is cofibrant and $Y\in\Ob\mathscr{C}$
  is fibrant.
  Suppose $j_i:\CQ Y\to \widetilde{Y_i}$, $i=1,2$ are objectwise weak
  equivalences and each $\widetilde{Y_i}$ is Reedy fibrant.  Then
  there is a zig-zag of weak equivalences in $\qSet$ joining
  $[X, \widetilde{Y_1}]$ to $[X,\widetilde{Y_2}]$.
\end{lemma}
\begin{proof}
  It is sufficient to show that $[X, f]$ is a weak equivalence in
  $\qSet$ if $f$ is an acyclic fibration joining Reedy-fibrant
  resolutions of $\CQ Y$ (this amounts to Ken Brown's lemma and some
  other standard model-category theoretic moves). It is tempting to
  conclude by arguing that $[X, {-}]$ is a right Quillen functor
  $q\mathscr{C} \to \qSet$.  Unfortunately, $[X,{-}]$ is not right
  Quillen as its left adjoint ${-} \otimes X$ does not preserve
  acyclic cofibrations of cubical sets.  Fortunately, we only need
  ${-}\otimes X$ to preserve cofibrations.  Let $i_n:\partial\square^n
  \to \square^n$ be the usual inclusion.  Observe that
  \begin{equation}\label{eq:reedyliftingconditions}
    a \pitchfork [X, f] \quad\text{if and only if}\quad
    a\odot (\emptyset\to X) \pitchfork f \quad\text{if and only
      if}\quad
    (\emptyset\to X) \pitchfork \langle a\backslash f\rangle.
  \end{equation}
  But since $q$ is an acyclic Reedy fibration, $\langle a\backslash
  f\rangle$ is an acyclic fibration in $\mathscr{C}$.  Since $X$ is
  cofibrant, the equivalent conditions
  \eqref{eq:reedyliftingconditions} hold, so $[X,f]$ is an acyclic
  fibration in $\qSet$.
\end{proof}

\begin{lemma}\label{lemma:rkcpreservesresolutions}
  Suppose $Y\in\Ob\mathscr{C}$ is fibrant and $g:\CS Y \to \overline{Y}$
  is a simplicial resolution of $Y$.  Then $R_{K,\mathscr{C}}g$ is a
  cubical resolution of $Y$.
\end{lemma}
\begin{proof}
  Note that $R_{K,\mathscr{C}} \CS Y \cong \CQ Y$ by Lemma
  \ref{lemma:rightadjointpreservesconstantdiagrams} and
  $R_{K,\mathscr{C}}\overline{Y}$ is Reedy fibrant since
  $R_{K,\mathscr{C}}$ is right Quillen (Proposition
  \ref{prop:reedyinheritsleftquillen}).  What we need to check is that
  $R_{K,\mathscr{C}} g$ is an objectwise weak equivalence.  Let
  $q:\square^0\to\square^n$ be any inclusion.  Consider
  the square
  \begin{equation*}
    \xymatrix@C=4pc{(R_{K,\mathscr{C}} \CS Y)^{\square^n}
      \ar[r]^{(R_{K,\mathscr{C}} g)^{\square^n}}
      \ar[d]_{q^\ast} & (R_{K,\mathscr{C}} \overline{Y})^{\square^n}
      \ar[d]^{q^\ast} \\
      (R_{K,\mathscr{C}} \CS Y)^{\square^0}
      \ar[r]_{(R_{K,\mathscr{C}} g)^{\square^0}}
      & (R_{K,\mathscr{C}} \overline{Y})^{\square^0}.
    }
  \end{equation*}
  By our computations, this is isomorphic to
  \begin{equation*}
    \xymatrix@C=4pc{Y \ar@{=}[d] \ar[r] & \overline{Y}^{|\square^n|}
      \ar[d]^{|q|^\ast} \\
      Y \ar[r] & \overline{Y}_0.
    }
  \end{equation*}
  The bottom arrow is a weak equivalence since $\CS Y \to
  \overline{Y}$ is a resolution; the arrow $|q|^\ast$ is a weak
  equivalence since $|q|$ is an acyclic cofibration and $\overline{Y}$
  is homotopically constant.  Hence the top arrow is a weak
  equivalence.
\end{proof}

Recall that Theorem \ref{theorem:quillenequivalences} gives a triangle
of Quillen equivalences
\begin{equation*}
  \xymatrix{\qSet \ar[rr]^{i_!} \ar[rd]_{|{-}|} &&
    \qsSet \ar[ld]^{|{-}|_\Sigma} \\
    & \sSet.}
\end{equation*}
Let's write $\Sing$ and $\Sing_\Sigma$ for the right adjoints
of $|{-}|$ and $|{-}|_\Sigma$, respectively.
\begin{theorem}\label{theorem:qsigmaenrichment}
  Suppose $\mathscr{C}$ is a $\qsSet$-model category and
  that $X$ and $Y$ are cofibrant and fibrant objects of $\mathscr{C}$,
  respectively.  There is a zig-zag of natural weak equivalences
  joining $\Sing_\Sigma \mathbf{F}(X,Y)$ to $\mathscr{C}(X, Y)$.
\end{theorem}
\begin{proof}
  Define a cubical object $\widetilde{Y} : \QQ^\Op \to \mathscr{C}$ by
  $\widetilde{Y}_{\sqparen{n}} = Y^{\square_\Sigma^n}$.  Then
  $\widetilde{Y}$ is manifestly Reedy fibrant: the map
  $\widetilde{Y}(\sqparen{n}) \to M_{\sqparen{n}} \widetilde{Y}$ is
  the fibration $Y^{\square_\Sigma^n} \to
  Y^{\partial\square_\Sigma^n}$.  Moreover, the map $\CQ Y \to
  \widetilde{Y}$ induced by the projection maps $Y^{\square^0_\Sigma}
  \to Y^{\square^n_\Sigma}$ is a weak equivalence.  Hence $\CQ Y \to
  \widetilde{Y}$ is a cubical resolution of $Y$.  Choose a simplicial
  resolution $g:\CQ Y \to \overline{Y}$.  By Lemma
  \ref{lemma:rkcpreservesresolutions}, $R_{K,\mathscr{C}} g$ is a
  cubical resolution of $Y$, so there is a zig-zag of weak
  equivalences joining the cubical sets $[X, R_{K,\mathscr{C}}
  \overline{Y}]$ and $[X, \widetilde{Y}]$.  Observe that
  \begin{align*}
    \qSet(A, [X, R_{K,\mathscr{C}} \overline{Y}]) &\cong
    q\mathscr{C}(A\otimes X, R_{K,\mathscr{C}} \overline{Y}) \\
    &\cong
    s\mathscr{C}(|A|\otimes X, \overline{Y}) \\
    & \cong \sSet (|A|, [X,\overline{Y}])
  \end{align*}
  is natural in $A\in\qSet$,
  so there is an isomorphism
  \begin{equation*}
    [X, R_{K,\mathscr{C}}\overline{Y}] \cong \Sing [X,
    \overline{Y}] \cong i^\ast \Sing_{\Sigma} [X,\overline{Y}].
  \end{equation*}
  Now, $[X,\widetilde{Y}] \cong i^\ast \mathscr{C}(X,Y)$, so $i^\ast
  \mathscr{C}(X, Y)$ and $i^\ast \Sing_\Sigma [X, \overline{Y}]$ are
  weakly equivalent.  Hence $\mathscr{C}(X,Y)$ and $\Sing_\Sigma
  \mathbf{F}(X, Y)$ are weakly equivalent.
\end{proof}

\begin{remark}
  We've used the $\qsSet$ enrichment in order for the
  notational convenience.  However, the assiduous reader can check
  that if $\mathscr{C}$ has functorial path objects, there is a
  natural cotensor functor $Y^A$ with $A\in\qSet$,
  $X\in\mathscr{C}$ so that $Y^{-}$ sends (acyclic) cofibrations to
  (acyclic) fibrations for fibrant $Y$.  The proof of Theorem
  \ref{theorem:qsigmaenrichment} indicates that the mapping space
  given by $\mathscr{C}(X,Y)_n = \Hom_\mathscr{C}(X, Y^{\square^n})$
  has the correct homotopy type.
\end{remark}

\section{Enrichments for symmetric monoidal model categories}\label{section:existence}
In Section \ref{section:enriched}, we gave a criterion for a symmetric
monoidal model category $\mathscr{C}$ to have a extended cubical
enrichment, namely that it possess a cubical monoid satisfying some
homotopical properties.  In this section, we'll show that all
combinatorial symmetric monoidal model categories with cofibrant unit
have extended cubical enrichments.

\begin{theorem}\label{theorem:final}
  Suppose $(\mathscr{C},\otimes,e,[{-},{-}])$ is a combinatorial
  symmetric monoidal model category satisfying the
  Schwede-Shipley monoid axiom \cite{schwedeshipley}.  Suppose further
  that $\emptyset\to e$ is a cofibration.  Then $\mathscr{C}$ has a
  cubical monoid
  \begin{equation*}
    \vcenter{\xymatrix{
        e \amalg e \ar[r]^{d_0 \amalg d_1} \ar@/_/[rr]_{\id\amalg\id} &
        I \ar[r]^s & e
      }}
  \end{equation*}
  so that $d_0\amalg d_1$ is a cofibration and $s$ a weak
  equivalence.    
\end{theorem}
\begin{proof}
  Recall that $[1] = \{ 0 < 1\}$ has a symmetric monoidal structure given by
  conjunction.  We identify $[1]$ with the free $\mathscr{C}$-category
  it generates; then by Proposition \ref{prop:enrichedomnibus},
  $\mathscr{C}^{[1]}$ is a closed symmetric monoidal model category
  with the convolution product.  In this case, a map $i \to j$ of
  arrows 
  \begin{equation*}
    \vcenter{\xymatrix{A \ar[r]^f \ar[d]_i & C \ar[d]^j \\
      B \ar[r]_g & D}}
  \end{equation*}
  is a cofibration if and only if both $f$ and $g \amalg j : B
  \amalg_A C \to D$ are cofibrations.  In particular, note that
  $\emptyset\to e$ is cofibrant in $\mathscr{C}^{[1]}$.  Now consider
  the category $\Cat{Mon}(\mathscr{C}^{[1]})$ of monoids in
  $\mathscr{C}^{[1]}$.  This admits a model structure by \cite[Theorem
  4.1 (3)]{schwedeshipley} lifted from the projective model structure
  on $\mathscr{C}^{[1]}$.  One important property of the model
  structure on $\Cat{Mon}(\mathscr{C}^{[1]})$ is that a cofibration
  whose source is cofibrant in $\mathscr{C}^{[1]}$ is a cofibration in
  $\mathscr{C}^{[1]}$ \emph{(loc.~cit.)}.  We may thus produce a
  factorization
  \begin{equation}
    \vcenter{\xymatrix{
      \emptyset \ar[d] \ar[r] & X \ar[r]^f \ar[d]_j & e \ar[d]^{\id_e} \\
      e \ar[r]_{r} & Y \ar[r]_g & e
    }}
  \end{equation}
  of monoids in which $f$ and $g$ are weak equivalences, $X$ is
  cofibrant, and $j\amalg r:X \amalg e \to Y$ is a cofibration.  Now consider
  this as a diagram in $\mathscr{C}$.  We take the pushout of $f$ and $j$:
  \begin{equation}
    \vcenter{\xymatrix{
      \emptyset \ar[d] \ar[r] & X \ar[r]^f \ar[d]_j & e \ar[d]^{d_0}
      \ar[r]^{\id_e} & e \ar[d]^{\id_e} \\
      e \ar@/_/[rr]_{d_1} \ar[r]^{r} & Y \ar[r]^\pi & I \ar[r]_{s} & e.
    }}
  \end{equation}
  Let $d_1 = \pi r$.  We have two things to verify:
  \begin{enumerate}
  \item \emph{$I$ is a cubical monoid.}  By Proposition
    \ref{prop:reidbarton}, it is sufficient to show that $e\to I$ is a
    monoid in $\mathscr{C}^{[1]}$, that $d_1$ is the unit, and
    that $(\id_e, s)$ is a monoid map.  This is entirely formal.
  \item \emph{$d_0 \amalg d_1$ is a cofibration and $s$ a weak
      equivalence.} First, note $j$ is a cofibration with cofibrant
    source and $f$ a weak equivalence, so $\pi$ is a weak equivalence
    \cite[Proposition 13.1.2]{hirschhorn} (see also \cite{reedy}).
    Since $s\pi = g$ is a weak equivalence, $s$ is a weak equivalence.
    Now observe that
      \begin{equation*}
        \vcenter{\xymatrix{
            X \amalg e \ar[d]_{f \amalg \id_e} \ar[r]^{j\amalg r} &
            Y \ar[d]^\pi \\
            e \amalg e \ar[r]_{d_0\amalg d_1} & I
          }}
      \end{equation*}
      is a pushout square; since $j\amalg r$ is a cofibration,
      $d_0\amalg d_1$ is a cofibration as well.\qedhere
  \end{enumerate}
\end{proof}

\bibliographystyle{amsplain}
\bibliography{master}

\end{document}